\documentclass{article}
\usepackage[utf8]{inputenc}
\usepackage[T1]{fontenc}
\usepackage{amsmath,amssymb}
\usepackage{graphicx}
\usepackage{amsthm, enumerate}
\theoremstyle{remark}
\usepackage{hyperref}
\usepackage{latexsym}
\usepackage{xypic}
\usepackage[all]{xy}
\usepackage{pxfonts}
\usepackage{mathrsfs}
\usepackage{fullpage}
\usepackage{wasysym}
\usepackage[affil-it]{authblk}

\def\deltaII{\normalsize\blacktriangle\mspace{-12mu}\mbox{\Large$\bigtriangleup$}}

\def\car{ \square }

\def\tc{ \mathcal{A}ct }
\def\sc0{ \mathcal{A}ct_{>0} }

\def\pent{ $\begin{large}$
\pentagon
$\end{large}$ } 

\def\penta{ $\begin{large}$
\tilde{\pentagon}
$\end{large}$ }

\theoremstyle{definition}
\newtheorem{defi}{Definition}[section]
\newtheorem{const}[defi]{Construction}
\newtheorem{expl}[defi]{Example}
\newtheorem{app}[defi]{Application}

\newtheorem*{conv}{Convention}
\newtheorem*{plan}{Organization of the paper}
\newtheorem*{merci}{Acknowlegments}

\theoremstyle{plain}
\newtheorem{pro}[defi]{Proposition}
\newtheorem{thm}[defi]{Theorem}

\newtheorem{cor}[defi]{Corollary}

\newtheorem{lmm}[defi]{Lemma}

\theoremstyle{remark}

\newtheorem{rmk}[defi]{Remark}
\newtheorem{notat}[defi]{Notation}
\newtheorem*{sproof}{Sketch of proof}

\title{Swiss-cheese action on the totalization of operads under the monoid actions operad}
\author{Julien Ducoulombier}
\date{}

\begin{document}
\maketitle 
\begin{abstract}
\noindent
We prove that if a pair of semi-cosimplicial spaces $(X_{c}^{\bullet}\,;\, X^{\bullet}_{o})$ arise from a coloured operad then the semi-totalization $sTot(X_{o}^{\bullet})$ has the homotopy type of a relative double loop space and the pair $(sTot(X_{c}^{\bullet})\,;\,sTot(X_{o}^{\bullet}))$ is weakly equivalent to an explicit algebra over the two dimensional Swiss-cheese operad $\mathcal{SC}_{2}$.
\end{abstract}

\section*{Introduction}

A multiplicative operad $O$ is an operad under the associative operad $As$. In 
\cite{MR2061854} McClure and Smith build a cosimplicial space $O^{\bullet}$ from the multiplicative operad $O$ and show that, under some conditions, its homotopy totalization is a double loop space. V. Turchin in \cite{Tourtchine:arXiv1012.5957} and independently Dwyer and Hess in \cite{2010arXiv1006.0874D} are able to identify the space of double delooping and prove, under the assumption : $O(0)\simeq O(1) \simeq \ast$, that 
\begin{center}
$
hoTot(O^{\bullet})\simeq \Omega^{2} Operad^{h}(As\,;\,O),
$
\end{center}
where $Operad^{h}(As\,;\,O)$ is the space of derived maps from the associative operad to $O$.

In order to prove this statement, V. Turchin introduces the categories of bimodules and infinitesimal bimodules over an operad $O$, denoted respectively by $Bimod_{O}$ and $Ibimod_{O}$, such that $hoTot(O^{\bullet})$ is weakly equivalent to $Ibimod_{As}^{h}(As\,;\,O)$. Then he proves the following two weak equivalences:
\begin{center}
$
Ibimod_{As}^{h}(As\,;\,O)\simeq \Omega Bimod_{As}^{h}(As\,;\,O) \,\,\,\,
$ and 
$
\,\,\,\,
Bimod_{As}^{h}(As\,;\,O)\simeq \Omega Operad^{h}(As\,;\,O).
$
\end{center}

This result was motivated by the following theorem of D. Sinha: the space of long knots $Emb_{c}(\mathbb{R}\,;\,\mathbb{R}^{d})$ has the homotopy type of $hoTot(\mathcal{K}_{d}^{\bullet})$ where $\mathcal{K}_{d}$ is a multiplicative operad weakly equivalent to the little $d$-disk operad. The \textit{Swiss-cheese} operad $\mathcal{SC}_{d}$ is a relative version of the little disc operad. It is a two coloured topological operad with set of colours $S=\{o\,;\,c\}$ that has been introduced by A. Voronov in \cite{VoronovSC}. In particular, if $f:A\rightarrow X$ is a pointed continuous map then the following pair is an $\mathcal{SC}_{d}$-space:
\begin{center}
$
\big(\Omega^{d} X\,;\, \Omega^{d}(X\,;\,A)\big)\, :=\, \big( \Omega^{d} X \,;\, hofib(\Omega^{d-1}A\rightarrow \Omega^{d-1}X)\big). 
$ 
\end{center}

In this paper we make great use of the operad $\pi_{0}(\mathcal{SC}_{1})$ which is the operad of monoid actions $\tc$: it is a $2$-coloured operad whose algebras are the pairs of spaces $(X\,;\,A)$ where $X$ is a monoid and $A$ a left $X$-module. The operad $\sc0$ is the non-unital version of $\tc$. Similarly to the uncoloured case there is a notion of $\sc0$-
bimodule and $\sc0$-infinitesimal bimodule. We prove that if $O$ is an 
operad under $\tc$ then it gives rise to a pair of semi-cosimplicial spaces 
$(O_{c}\,;\, O_{o})$ such that the pair $(sTot(O_{c}
)\,;\,sTot(O_{o}))$ is weakly equivalent to:
\begin{center}
$
\left(
\Omega^{2} Operad^{h}(As_{>0}\,;\,O_{c})\,;\, \Omega^{2}\big( Operad^{h}(As_{>0}\,;\,O_{c})\,;\, Operad^{h}(\sc0\,;\,O)\big)
\right).
$ 
\end{center}
that is, an $\mathcal{SC}_{2}$-space.

\hspace{-300cm}\footnote{Universit\'e Paris 13, Sorbonne Paris Cit\'e, LAGA, CNRS, UMR 7539, 93430 \mbox{Villetaneuse}, France}
\footnote{Key words: coloured operads, loop spaces, cosimplicial spaces, model categories}
\footnote{ducoulombier@math.univ-paris13.fr}

\begin{plan}
The paper is divided into six sections. The first one is an introduction. It describes the categories of coloured operads, bimodules and infinitesimal bimodules over an operad. An explicit description of a point $X$ in $Bimod_{\sc0}$ and $Ibimod_{\sc0}$ in terms of pairs of semi-cosimplicial spaces $(X_{c}\,;\,X_{o})$ is given. We insist on the link between bimodule structures over $\sc0$ and monoidal structures on  semi-cosimplicial spaces introduced by McClure and Smith in \cite{MR2061854}.

The second section introduces the left adjoint functors to the forgetful functors from the categories of bimodules and infinitesimal bimodules over an $S$-coloured operad to the category of $S$-sequences. These adjunctions will be used in the third section in order to define a model category structure on $Bimod_{O}$ and $Ibimod_{O}$.  We also determine an explicit cofibrant replacement of $\tc$ (resp. $\sc0$) in the model category $Ibimod_{\sc0}$ (resp. $Bimod_{\sc0}$) and prove the weak equivalence:
$$
Ibimod_{\sc0}^{h}(\tc\,;\,M)\simeq Ibimod_{As_{>0}}^{h}(As\,;\,M_{c}),
$$
where $M$ is an $\sc0$-infinitesimal bimodule and $M_{c}$ is its closed part.

In section four we prove the first relative delooping theorem. From an $\sc0$-bimodule map $\eta:\tc\rightarrow M$ we extract two semi-cosimplicial spaces $(M_{c}\,;\,M_{o})$. We prove, under some conditions,  the weak equivalence of pairs: 
$$
\big(sTot(M_{c})\,;\,sTot(M_{o})\big)\simeq \big( \Omega Bimod_{As_{>0}}^{h}(As_{>0}\,;\,M_{c})\,;\, \Omega \big( Bimod_{As_{>0}}^{h}(As_{>0} \,;\, M_{c})\,;\, Bimod_{\sc0}^{h}(\sc0 \,;\, M)\big)\,\big).
$$

Section five consists in considering a particular case where a double relative delooping theorem holds. Namely, let $\alpha : As\rightarrow O$ be a map of operads and $\beta:O\rightarrow B$ be a map of $O$-bimodules. The two objects $O$ and $B$ are equipped with semi-cosimplicial structures. Under some conditions, we prove the following weak equivalence of pairs:
$$
\big(sTot(O)\,;\,sTot(B)\big)\simeq\big( \Omega^{2} Operad^{h}(As_{>0}\,;\,O)\,;\, \Omega^{2}\big( Operad^{h}(As_{>0}\,;\,O)\,;\,Operad_{\{o\,;\,c\}}^{h}(\sc0\,;\,X)\big)\,\big),
$$
where $X$ is a coloured operad build out of $O$ and $B$.

The last section is devoted to the proof of the main theorem: if $O$ is an $\{o\,;\,c\}$-operad under $\tc$ such that $O(0\,;\,c)\simeq O(1\,;\,c)\simeq O(1\,;\,o)\simeq \ast$ then we have the weak equivalence of pairs: 
$$
\big(sTot(O_{c})\,;\,sTot(O_{o})\big)\simeq\big( \Omega^{2} Operad^{h}(As_{>0}\,;\,O_{c})\,;\, \Omega^{2}\big( Operad^{h}(As_{>0}\,;\,O_{c})\,;\,Operad_{\{o\,;\,c\}}^{h}(\sc0\,;\,O)\big)\,\big).
$$
\end{plan}

\begin{conv}
By space we mean compactly generated Hausdorff space and by abuse of notation we denote by $Top$ this category (see e.g. \cite{MR1650134} section 2.4). If $X$, $Y$ and $Z$ are spaces then $Top(X;Y)$ is equipped with the compact-open topology in order to have a homeomorphism $Top(X;Top(Y;Z))\cong Top(X\times Y;Z)$.

A \textit{semi-cosimplicial space} $X^{\bullet}$ is a family of topological spaces $\{X^{n}\}_{n\geq 0}$ endowed with operations,
\begin{center}
 $d^{i}:X^{n}\rightarrow X^{n+1},\,\,\,\,\,$ for $i\in \{0,\ldots,n+1\}$,
\end{center}
satisfying the cosimplicial relations: $d^{j}d^{i}=d^{i}d^{j-1}$ for $0\leq i<j$. By \textit{semi-totalization} $sTot(X^{\bullet})$ we mean the space of natural transformations from the semi-cosimplicial space $\Delta^{\bullet}$ to $X^{\bullet}$. The semi-totalization is also called fat-totalization and it is a homotopy invariant. Since the homotopy totalization is weakly equivalent to the semi-totalization \cite[Lemma 3.8]{MR0458417}, we will  ignore the codegeneracies in the present work. We denote weak equivalences by the symbol $\simeq$.
\end{conv}

\begin{merci}
I would like to thank my PhD advisor Muriel Livernet for her input on this project.
I thank also Victor Turchin for his interest in this project and fruitful conversations.
\end{merci}


\newpage
\section{Bimodules and infinitesimal bimodules over a coloured operad}

In what follows we introduce the category of coloured operads as well as the  categories of bimodules and infinitesimal bimodules over a coloured operad. 
We focus on the operads with two colours $\{o\,;\,c\}$ called $\{o\,;\,c\}$-operads. In particular we define the $\{o\,;\,c\}$-operad $\sc0$ of monoid actions as in \cite{2013arXiv1312.7155H}. Besides, we characterize the bimodules and infinitesimal bimodules over this operad in terms of semi-cosimplicial spaces.

\subsection{The operad of (unital) monoid actions}

\begin{defi}
Let $S$ be a set. An \textit{$S$-sequence} is a collection of topological spaces $\{O(s_{1},\ldots,s_{n}; s_{n+1})\}_{s_{i}\in S}^{n\in \mathbb{N}}$. The set $S$ is called the \textit{set of colours}. A map between two $S$-sequences $O_{1}$ and $O_{2}$ is a collection of continuous maps: $$\{f_{s_{1},\ldots,s_{n};s_{n+1}}:O_{1}(s_{1},\ldots,s_{n}; s_{n+1})\rightarrow O_{2}(s_{1},\ldots,s_{n}; s_{n+1})\}_{s_{i}\in S}^{n\in \mathbb{N}}.$$ We denote by $Coll(S)$ the category of $S$-sequences. 
\end{defi}

\begin{notat}
If $M$ is an $\{o\,;\,c\}$-sequence, then we use the following notation in the rest of the text:
\begin{center}
$ M_{c}^{n}=M(n\,;\,c)=M(\underset{n}{\underbrace{c,\ldots,c}};c)\,\,\,\,\,$ and
$ \,\,\,\,\, M_{o}^{n}=M(n+1\,;\,o)=M(\underset{n}{\underbrace{c,\ldots,c}},o;o)$.
\end{center} 
We denote by $M_{c}$ the family $\{M_{c}^{n}\}_{n\geq 0}$ and by $M_{o}$ the family $\{M_{o}^{n}\}_{n\geq 0}$.
\end{notat}

\begin{defi} 
An \textit{$S$-operad} is an $S$-sequence $O$ endowed with operations:\\{}\\
$\circ_{i}:O(s_{1},\ldots,s_{n};s_{n+1})\times O(s'_{1},\ldots,s'_{m};s_{i})
\rightarrow O(s_{1},\ldots,s_{i-1},s'_{1},\ldots,s'_{m},s_{i+1},\ldots,s_{n};s_{n+1}),\,\,\,\,\,$ for $1\leq i\leq n$,\\{}\\
and distinguished elements $\{\ast_{s} \in O(s;s)\}_{s\in S}$
satisfying associativity and unit axioms \cite{Arone:arXiv1105.1576}. We denote by $x\circ_{i}y$ the operation $\circ_{i}(x\,;\, y)$ for $x,y\in O$.
Define $Operad_{S}$ to be the category of $S$-operads where a map of $S$-operads is an $S$-sequence map which preserves the operadic structure.

Let $O$ be an $S$-operad and $A=\{A_{s}\}_{s\in S}$ be a family of topological spaces. The \textit{endomorphism $S$-operad} $End_{A}$ (see \cite{MR2342815}) is the family of spaces of continuous maps defined by:
$$
End_{A}(s_{1},\ldots,s_{n};s_{n+1})=Top(A_{s_{1}}\times \ldots \times A_{s_{n}};A_{s_{n+1}}).
$$
The family $A$ is called an \textit{$O$-space} if there exists a map of $S$-operads $O\rightarrow End_{A}$.
\end{defi}

\begin{defi}{\cite{2013arXiv1312.7155H}}\label{AM}
$\,$ Let $S=\{o,c\}$. The \textit{$S$-operad of monoid actions} $\sc0$ is given by the $S$-sequence: 
\begin{center}
$\sc0(n\,;\,c) = \ast_{n\,;\,c}\,\,\,\,$ for $n>0$, $\,\,\,\,\,\,\,\,\,\,\,\,\sc0(n\,;\,o) = \ast_{n\,;\,o}\,\,\,\,$ for $n>0$,  
\end{center}
and the empty set otherwise with $\ast_{n\,;\,c}$ and $\ast_{n\,;\,o}$ being the one point topological space. The compositions are the following:

\begin{equation}
\ast_{n+m-1\,;\,c}=\ast_{n\,;\,c}\circ_{i}\ast_{m\,;\,c}\,\,\,\,\,\,\, \text{and}  
\,\,\,\,\,\,\,\ast_{n+m-1\,;\,o}=
\left\{\begin{array}{cc}
\ast_{n\,;\,o}\circ_{i}\ast_{m\,;\,c}, &  \text{for}\,\,  i\neq n,\\ 
\ast_{n\,;\,o}\circ_{n}\ast_{m\,;\,o}. & 
\end{array} \right.  
\end{equation}

Similarly \textit{the $S$-operad of unital monoid actions} $\tc$ is given by the $S$-sequence:
\begin{center}
$\sc0(n\,;\,c) = \ast_{n\,;\,c}\,\,\,\,$ for $n\geq 0$, $\,\,\,\,\,\,\,\,\,\,\,\,\sc0(n\,;\,o) = \ast_{n\,;\,o}\,\,\,\,$ for $n>0$,  
\end{center}
and the empty set otherwise with the same compositions. Consequently, the $S$-operad $\tc$ (resp. $\sc0$) is generated by $\ast_{0\,;\,c}$, $\ast_{2\,;\,c}$ and $\ast_{2\,;\,o}$ (resp. $\ast_{2\,;\,c}$ and $\ast_{2\,;\,o}$).

An \textit{$\tc$-space} is a pair of topological spaces $(X;A)$ with $X$ a topological monoid with unit and $A$ a left module over $X$.

The $\{c\}$-sequence given by the restriction of $\tc$ (resp. $\sc0$) to the colour $\{c\}$ is the \textit{associative operad} $As$ (resp. the \textit{strict associative operad} $As_{>0}$). We use the notation $\ast_{n}$ to refer to the one point topological space $As(n)$.
\end{defi}

The operad of monoid actions has been introduced by Hoefel, Livernet and Stasheff in  \cite{2013arXiv1312.7155H} in the context of recognition principle for relative loop space.

\subsection{Infinitesimal bimodules over a coloured operad}

\begin{defi}\label{bimodule}
Let $O$ be an $S$-operad. An \textit{infinitesimal bimodule} over the operad $O$ (or $O$-infinitesimal bimodule) is an $S$-sequence $M$ endowed with operations: 
\begin{itemize}
\item[$\circ_{i}:$] $O(s_{1},\ldots,s_{n};s_{n+1})\times M(s'_{1},\ldots,s'_{m};s_{i}) 
\rightarrow M(s_{1},\ldots,s_{i-1},s'_{1},\ldots,s'_{m},s_{i+1},\ldots,s_{n};s_{n+1})$, $\,\,\,$  for $1\leq i\leq n$,
\item[$\circ^{i}:$] $M(s_{1},\ldots,s_{n};s_{n+1})\times O(s'_{1},\ldots,s'_{m};s_{i})
\rightarrow M(s_{1},\ldots,s_{i-1},s'_{1},\ldots,s'_{m},s_{i+1},\ldots,s_{n};s_{n+1})$, $\,\,\,$  for $1\leq i\leq n$,
\end{itemize}
satisfying associativity and unit relations \cite{Arone:arXiv1105.1576}. A map between $O$-infinitesimal bimodules is given by an $S$-sequence map preserving this structure. Let $Ibimod_{O}$ be the category of infinitesimal bimodules over $O$. We denote by $x\circ_{i} y$ (resp. $x\circ^{i} y$) the operation $\circ_{i}(x\,;\,y)$ (resp. $\circ^{i}(x\,;\,y)$) with $x\in O$ and $y\in M$ (resp. $x\in M$ and $y\in O$).
\end{defi}

\begin{expl} 
For any $S$-operad map $\eta:O_{1}\rightarrow O_{2}$, $\,\,O_{2}$ is endowed with the following $O_{1}$-infinitesimal bimodule structure:
\begin{center}
$\circ_{i}: \,\,\, O_{1}\times O_{2} \overset{\eta \times id}{\rightarrow}  O_{2}\times O_{2} \overset{\circ_{i}}{\rightarrow}  O_{2}\,\,\,\,\,\,\,$ and
$ \,\,\,\,\,\,\,\circ^{i}:\,\,\, O_{2}\times O_{1} \overset{id\times \eta}{\rightarrow}  O_{2}\times O_{2} \overset{\circ_{i}}{\rightarrow}  O_{2}$.
\end{center}
Consequently, if $A$ is an $O$-space then $End_{A}$ is an $O$-infinitesimal bimodule.
\end{expl}

\begin{defi}
Let $N$ and $M$ be two $S$-sequences. The sequence $M$ \textit{is of type} $N$ if: $$N(s_{1},\ldots,s_{n};s_{n+1})=\emptyset  \Rightarrow M(s_{1},\ldots,s_{n};s_{n+1})=\emptyset.$$
\end{defi}

\begin{pro}\label{SCwbimodule}
Let $M$ be an $\{o\,;\,c\}$-sequence of type $\tc$. The following assertions are equivalent:
\begin{itemize}
\item[$i)$] $M$ is an $\sc0$-infinitesimal bimodule;
\item[$ii)$] the families $M_{c}$ and $M_{o}$ are semi-cosimplicial spaces and there exists a semi-cosimplicial map $h:M_{c}\rightarrow M_{o}$.
\end{itemize} 
Moreover $i)\Rightarrow ii)$ even if $M$ is not of type $\tc$.
\end{pro}

\begin{proof}
Let $M$ be an $\sc0$-infinitesimal bimodule. For $n\in \mathbb{N}$,  let $h:M_{c}^{n}\rightarrow M_{o}^{n}$ defined by  $h(x)=\ast_{2\,;\,o}\circ_{1} x$.
The semi-cosimplicial structure is given as usual (see e.g \cite{Arone:arXiv1105.1576} , \cite{MR2061854} and \cite{MR2188133}) by:
$$d^{i}:M_{c}^{n}\rightarrow M_{c}^{n+1};\,\, x\mapsto 
\left\{ \begin{array}{ll}
\ast_{2\,;\,c} \circ_{2} x, & $if$ \,\,\,\,\,i=0, \\ 
x\circ^{i}\ast_{2\,;\,c}, & $if$ \,\,\,\,\,i\in \{1,\ldots , n\}, \\ 
\ast_{2\,;\,c} \circ_{1} x, & $if$ \,\,\,\,\,i=n+1,
\end{array}\right.
\,\, \text{and} \,\,\,\,
d^{i}:M_{o}^{n}\rightarrow M_{o}^{n+1}; \,\,x\mapsto
\left\{ \begin{array}{lll}
\ast_{2\,;\,o} \circ_{2} x, & $if$ & i=0, \\ 
x\circ^{i}\ast_{2\,;\,c}, & $if$ & i\in \{1,\ldots , n\}, \\ 
x\circ^{n+1}\ast_{2\,;\,o}, & $if$ & i=n+1.
\end{array}\right. $$
The reader can check that the relations $(1)$ of Definition \ref{AM} and   Definition \ref{bimodule} induce the semi-cosimplicial relations.

Conversely, if $h:M_{c}\rightarrow M_{o}$ is a semi-cosimplicial map, let $M(n\,;\,c)=M_{c}^{n}$, $ M(n+1\,;\,o)=M_{o}^{n}$ and the empty set otherwise. The left and right infinitesimal module structures are defined by the above construction, since $\sc0$ is generated by $\ast_{2\,;\,c}$ and $\ast_{2\,;\,o}$ as a coloured operad.
\end{proof}

It is proved in \cite{Tourtchine:arXiv1012.5957} that the category of semi-cosimplicial spaces is equivalent to the category of $As_{>0}$-infinitesimal bimodules. Consequently the collection $M_{o}=\{M_{o}^{n}\}_{n\geq 0}$ is an infinitesimal bimodule over $As_{>0}$. Since $As_{>0}$ is generated by $\ast_{2}$ as an operad, the structure of $M_{o}$ is given by:
\begin{equation}
\left\{\begin{array}{ll}
\ast_{2}\circ_{2} x = \ast_{2\,;\,o}\circ_{2} x, &  \text{for}\,\, x\in M_{o}^{n},  \\ 
\ast_{2}\circ_{1} x = x\circ^{n+1} \ast_{2\,;\,o}, &  \text{for}\,\, x\in M_{o}^{n},  \\ 
x\circ^{i}\ast_{2} =  x\circ^{i}\ast_{2\,;\,c}, &  \text{for}\,\, x\in M_{o}^{n} \,\,\text{and}\,\, i\in \{1,\ldots,n\}. 
\end{array}\right.
\end{equation}


\subsection{Bimodules over a coloured operad}

\begin{defi}
Let $O$ be an $S$-operad. An $S$-sequence $M$ is an \textit{$O$-bimodule} if it is endowed with operations:\\{}\\
$\gamma_{l}:O(s_{1},\ldots,s_{n};s_{n+1})\times M(s^{1}_{1},\ldots,s^{1}_{p_{1}};s_{1}) \times \cdots \times M(s^{n}_{1},\ldots,s^{n}_{p_{n}};s_{n}) \rightarrow  M(s^{1}_{1},\ldots,s^{n}_{p_{n}};s_{n+1}),\,\,\,\,\,\,$ for $1\leq i\leq n$,\\{}\\
$\circ^{i}:M(s_{1},\ldots,s_{n};s_{n+1})\times O(s'_{1},\ldots,s'_{m};s_{i}) \rightarrow  M(s_{1},\ldots,s_{i-1},s'_{1},\ldots,s'_{m},s_{i+1},\ldots,s_{n};s_{n+1}),\,\,\,\,\,\,$ for $1\leq i\leq n$,\\{}\\
satisfying associativity and unit axioms \cite{Arone:arXiv1105.1576}. A map between $O$-bimodules is an $S$-sequence map which preserves the bimodule structure. Let $Bimod_{O}$ be the category of $O$-bimodules. We denote by $x(y_{1},\cdots,y_{n})$ the operation $\gamma_{l}(x,y_{1},\cdots,y_{n})$ with $x\in O$ and $y_{i}\in M$.
\end{defi}

\begin{expl}
For any $S$-operad map $\eta:O_{1}\rightarrow O_{2}$, $\,\,O_{2}$ is endowed with the following $O_{1}$-bimodule structure:
\begin{center}
$\gamma_{l}: O_{1}\times O_{2}\times \cdots \times O_{2}  \overset{\eta \times id \cdots id}{\longrightarrow}  O_{2}\times \cdots \times O_{2} \rightarrow  O_{2}\,\,\,\,\,\,$ and 
$\,\,\,\,\,\,\circ^{i}: O_{2}\times O_{1}  \overset{id\times \eta }{\longrightarrow}  O_{2}\times  O_{2} \rightarrow  O_{2}$.
\end{center}
Consequently, if $A$ is an $O$-algebra then $End_{A}$ is an $O$-bimodule.
\end{expl}

A priori there is no relation between an $O$-bimodule structure and an $O$-infinitesimal bimodule structure because the left operations differ. However, if $\eta: O\rightarrow M$ is a morphism of $O$-bimodules then $M$ is an $O$-infinitesimal bimodule and the left infinitesimal bimodule structure is given by:
\begin{center}
$
\begin{array}{cccc}
\circ_{i}\,: & O(s_{1},\ldots,s_{n};s_{n+1})\times M(s'_{1},\ldots,s'_{m};s_{i}) & \rightarrow & M(s_{1},\ldots,s_{i-1},s'_{1},\ldots,s'_{m},s_{i+1},\ldots,s_{n};s_{n+1}) \\ 
 & (o\,;\,m) & \mapsto & o\,\big( \eta(\ast_{s_{1}}),\ldots,\eta(\ast_{s_{i-1}}),m,\eta(\ast_{s_{i+1}}),\ldots,\eta(\ast_{s_{n}})\big)
\end{array} 
$
\end{center}
where $\ast_{s}$ is the distinguished element in $O(s;s)$.

In \cite{MR2061854} McClure and Smith define a monoidal structure on the category of semi-cosimplicial spaces in order to recognize loop spaces. More precisely, they prove that the group completion of the semi-totalization of a monoid in this category has the homotopy type of a loop space. We recall this construction since we need it to describe $\sc0$-bimodules under $\tc$.

\begin{pro}{\cite[proposition 2.2]{MR2061854}}\label{definMS}
Let $X^{\bullet}$ and $Y^{\bullet}$ be two semi-cosimplicial spaces and let $X \boxtimes Y$ be the semi-cosimplicial space whose $m$-th space is given by:
$$
\left( \underset{p+q=m}{\coprod} X^{p}\times X^{q}\right) /\sim
,$$
where $\sim$ is the equivalence relation generated by $(x,d^{0}y)\sim (d^{|x|+1}x,y)$.\\
The semi-cosimplicial structure is the following:
$$
d^{i}(x,y)=
\left\{
\begin{array}{ccc}
(d^{i}x,y), & if & 0\leq i\leq |x|, \\ 
(x,d^{i-|x|}y), & if & |x|< i \leq |x|+|y|+1.
\end{array} 
\right.
$$
The category of semi-cosimplicial spaces equipped with $\boxtimes$ is a monoidal category denoted by $(Top^{\Delta_{inj}}\,,\,\boxtimes)$, with unit $e$ being the constant semi-cosimplicial one point space.
\end{pro}

\begin{pro}\label{propSC}
Let $M$ be an $\{o\,;\,c\}$-sequence of type $\tc$. The following assertions are equivalent:
\begin{itemize}
\item[$i)$] $M$ is an $\sc0$-bimodule under $\tc$,
\item[$ii)$] in $(Top^{\Delta_{inj}}\,,\,\boxtimes)$ the family $M_{c}$ is a monoid with unit, the family $M_{o}$ is a $M_{c}$-left module and there exists a morphism of $M_{c}$-left module $h:M_{c}\rightarrow M_{o}$.
\end{itemize}
Moreover $i)\Rightarrow ii)$ even if $M$ is not of type $\tc$.
\end{pro}

\begin{proof}
Let $M$ be an $\sc0$-bimodule equipped with an $\sc0$-bimodule map $\eta:\tc \rightarrow M$. Let $M_{c}^{n}=M(n\,;\,c)$ and $M_{o}^{n}=M(n+1\,;\,o)$ for 
$n\in \mathbb{N}$. The bimodule structure induces the following cofaces:
\begin{equation}
\left\{
\begin{array}{llllll}
d^{i}&:M_{c}^{n}\rightarrow M_{c}^{n+1}&; & x\mapsto x\circ^{i}\ast_{2\,;\,c}, & \text{if} & i\in \{1,\ldots,n\}, \\ 
d^{i}&:M_{o}^{n}\rightarrow M_{o}^{n+1}&; & x\mapsto x\circ^{i}\ast_{2\,;\,c}, & \text{if} &i\in \{1,\ldots,n\}, \\ 
d^{n+1}&:M_{o}^{n}\rightarrow M_{o}^{n+1}&; & x\mapsto x\circ^{n+1}\ast_{2\,;\,o}, & &
\end{array} 
\right.
\end{equation}
satisfying the semi-cosimplicial relations and two operations:
\begin{equation}
\left\{
\begin{array}{lll}
M_{c}^{j}\times M_{c}^{l}\rightarrow M_{c}^{j+l}&; & (x;y)\mapsto \ast_{2\,;\,c}(x;y), \\ 
M_{c}^{j}\times M_{o}^{l}\rightarrow M_{o}^{j+l}&; & (x;y)\mapsto \ast_{2\,;\,o}(x;y).
\end{array} 
\right.
\end{equation}
The map $\eta:\tc \rightarrow M$ gives us the missing cofaces:
\begin{equation}
\left\{
\begin{array}{llll}
d^{0}&:M_{c}^{n}\rightarrow M_{c}^{n+1}&; & x\mapsto \ast_{2\,;\,c}( \eta(\ast_{1\,;\,c})\, ,\, x), \\ 
d^{n+1}&:M_{c}^{n}\rightarrow M_{c}^{n+1}&; & x\mapsto \ast_{2\,;\,c}( x \, ,\,\eta(\ast_{1\,;\,c})), \\ 
d^{0}&:M_{o}^{n}\rightarrow M_{o}^{n+1}&; & x\mapsto \ast_{2\,;\,o}(\eta(\ast_{1\,;\,c})\, ,\, x),
\end{array}
\right.
\end{equation}
inducing a  semi cosimplicial structure on $M_{c}$ and $M_{o}$ such that the two operations defined in $(3)$ make $M_{c}$ into a monoid with unit and $M_{o}$ into a $M_{c}$-left module. The map:
$$h:M_{c}^{n}\rightarrow M_{o}^{n};\,\, x\mapsto \ast_{2\,;\,o}(x\, , \, \eta(\ast_{1\,;\,o}))$$ 
is a left $M_{c}$-module map.

Conversely, let $(M_{c},M_{o},h)$ be a triple satisfying the conditions of the proposition. By using the same argument as in Proposition \ref{SCwbimodule}, the constructions $(3)$ and $(4)$ define an $\sc0$-bimodule structure on $M$. In particular, if  $M_{c}$ and $M_{o}$ coincide with the unit $e$, then the corresponding $\sc0$-bimodule is $\tc$. There exists a map $\eta_{c}$ from the unit to $M_{c}$, for  $M_{c}$ is a monoid with unit. Let $\eta_{o}$ be the map from the unit to $M_{o}$ given by $\eta_{o}=h\circ \eta_{c}$. The map $\eta:\tc \rightarrow M$ so obtained is an $\sc0$-bimodule map.
\end{proof}

This proposition implies that the category whose objects are monoids in $(Top^{\Delta_{inj}}\,,\,\boxtimes)$ is equivalent to the category of $As_{>0}$-bimodules under $As$ considered by Turchin. Furthermore if we substitute $\sc0$-bimodule by $\tc$-bimodule and semi-cosimplicial space by cosimplicial space, Proposition \ref{propSC} is still true. 

\begin{expl}\label{exemple}
Let $(X;\ast)$ be a pointed topological space and $A$ be a subspace of $X$ containing $\ast$. Let $\Omega X^{\bullet} $ and $\Omega(X;A)^{\bullet} $ be the two cosimplicial spaces defined respectively by: 
\begin{center}
$\Omega X^{n}:= X^{\times n}\,\,\,\,\,$ and $\,\,\,\,\,\Omega (X;A)^{n}:= X^{\times n}\times A\,\,\,\,\,$, for $n\in \mathbb{N}$ and 
\end{center}
$$
\begin{array}{llll}
d^{i}:\Omega X^{n}\rightarrow \Omega X^{n+1} & ; & (x_{1},\ldots,x_{n})&\mapsto
\left\{
\begin{array}{lll}
(\ast,x_{1},\ldots,x_{n}), & $ if $ &\,\,\,\,i=0, \\ 
(x_{1},\ldots,x_{i},x_{i},\ldots,x_{n}), & $ if $ &\,\,\,\,i\in\{1,\cdots,n\}, \\ 
(x_{1},\ldots,x_{n},\ast), & $ if $ &\,\,\,\,i=n+1,
\end{array} 
\right. \\ 
d^{i}:\Omega (X;A)^{n}\rightarrow \Omega (X;A)^{n+1} & ; & (x_{1},\ldots,x_{n},a)&\mapsto
\left\{
\begin{array}{lll}
(\ast,x_{1},\ldots,x_{n},a), & $ if $ &i=0, \\ 
(x_{1},\ldots,x_{i},x_{i},\ldots,x_{n},a), & $ if $ &i\in\{1,\cdots,n\}, \\ 
(x_{1},\ldots,x_{n},a,a), & $ if $ &i=n+1.
\end{array} 
\right.
\end{array} 
$$
The codegeneracies consist in forgetting a point and the concatenation makes $\Omega X^{\bullet} $ into a monoid with unit in $(Top^{\Delta_{inj}}\,,\,\boxtimes)$ and $\Omega(X;A)^{\bullet} $ into a left $\Omega X^{\bullet} $-module. The left $\Omega X^{\bullet}$-module map is 
defined by:
\begin{center}
$h:\Omega X^{n}\rightarrow \Omega(X;A)^{n};\,\, (x_{1},\ldots,x_{n})\mapsto (x_{1},\ldots,x_{n},\ast).$
\end{center}
Proposition \ref{propSC} states that these data are equivalent to an $\tc$-bimodule map.
The evaluation maps:
\begin{center}
$
\begin{array}{lll}
\Omega X\rightarrow Tot(\Omega X^{\bullet}) & ; & 
f\mapsto \left\{ f_{n}:(t_{1}\leq \cdots \leq t_{n}) \mapsto  \big( f(t_{1}),\ldots ,f(t_{n})\big)\right\}_{n} \\ 
\Omega(X\,;\,A)\rightarrow Tot(\Omega(X\,;\,A)^{\bullet}) & ; & f\mapsto \left\{ f_{n}:(t_{1}\leq \cdots \leq t_{n}) \mapsto  \big( f(t_{1}),\ldots ,f(t_{n}),f(1)\big)\right\}_{n}
\end{array} 
$
\end{center}
induce homeomorphisms. 
It provides an example of an $\tc$-bimodule map $\eta:\tc\rightarrow M$ such that the totalization of $M_{c}$  (resp. $M_{o}$) can be described as a loop space $\Omega X$ (respectively a relative loop space $\Omega(X;A)$) with explicit topological spaces $X$ and $A$. We will prove that we can generalize this result for any  $\sc0$-bimodule map $\eta:\tc \rightarrow M$ using the semi-totalization. 
\end{expl}

\section{The free (infinitesimal) bimodule generated by an $S$-sequence}

In what follows $S$ is a set, $O$ is an $S$-operad and $M$ is an $S$-sequence. In order to prove that $sTot(M_{o})$ has the homotopy type of a relative loop space and to identify explicitly this space we have to introduce a model category structure on the categories $Ibimod_{O}$ and $Bimod_{O}$. The easiest way is to use a transfer theorem (see e.g Theorem \ref{transfer}) which needs a left adjoint to the forgetful functor from the category of (infinitesimal) bimodules over $O$ to $Coll(S)$. In both cases, the first step consists in introducing the category of trees which encodes the (infinitesimal) bimodule structure. Then we label the vertices by points in $M$ or $O$. Similar constructions have been considered in \cite{BVlecturenote73} and more recently \cite{MR1996461}.\\{}\\ 
By a tree we mean a planar rooted tree with an orientation towards the root. Let $t$ be a tree:
\begin{itemize}
\item The set of its vertices is denoted by $V(t)$ and the set of its edges by $E(t)$. 

\item For a vertex $v$, the ordered set of its input edges is denoted by $in(v)$ and its cardinality by $|v|$ such that $in(v)=\{e_{1}(v),\ldots,e_{|v|}(v)\}$. The output edge of $v$ is denoted by $e_{0}(v)$.

\item The edges connecting two vertices are called \textit{inner edges} and the set of inner edges is denoted by $E^{int}(t)$.

\item An element $e\in E^{int}(t)$ is determined by a source vertex $s(e)$ and a target vertex $t(e)$ induced by the orientation of the tree. 

\item An edge with no source is called a \textit{leaf} and the ordered set of leaves is denoted by $\{l_{1},\ldots,l_{n}\}$. 

\item The edge with no target is called the \textit{trunk}, denoted by $e_{0}$, and its source, the \textit{root}, is denoted by $r$.

\item Each leaf is connected to the trunk by a unique path composed of  edges.

\item An \textit{$S$-tree} is a pair $(t,f)$ where $t$ is a planar tree and  $f:E(t)\rightarrow S$ is called an \textit{$S$-labelling} of $t$.
\end{itemize}

\begin{figure}[!h]
\begin{center}
\includegraphics[scale=0.3]{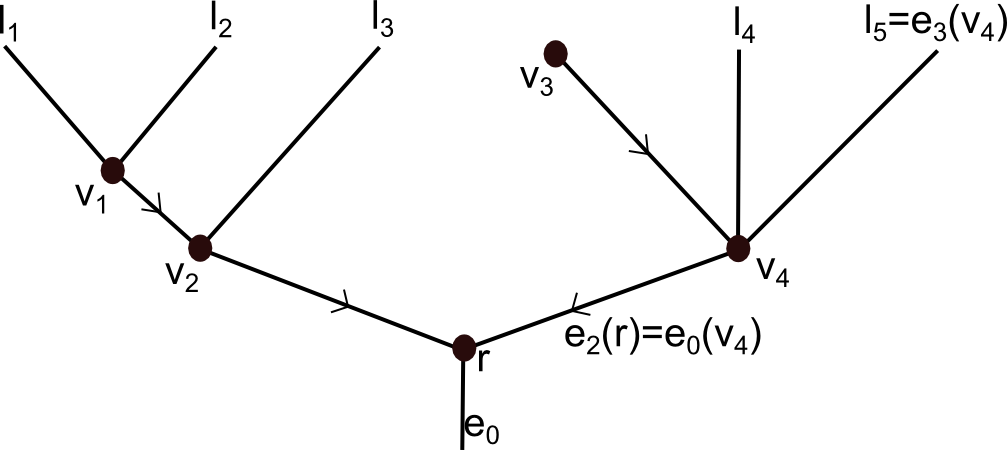}
\caption{A planar tree; $r$ is the root, $e_{0}$ is the trunk, $l_{1}$ is a leaf.}
\end{center}
\end{figure}

\subsection{The free infinitesimal bimodule}

\begin{defi}
The trees encoding the infinitesimal bimodule structure are constructed  as follows:
\begin{itemize}
\item The \textit{join} $j(v_{1}\,;\,v_{2})$ of two vertices $v_{1}$ and $v_{2}$ is the first common vertex shared by the two paths joining $v_{1}$ and $v_{2}$ to the root. If $j(v_{1}\,;\,v_{2})=r$, then $v_{1}$ and $v_{2}$ are said to be \textit{connected to the root} and if $j(v_{1}\,;\,v_{2})\in \{ v_{1};v_{2}\}$, then they are said to be \textit{connected}. In Figure $1$ the vertices
$v_{1}$ and $v_{2}$ are connected whereas the vertices $v_{1}$ and $v_{3}$ are connected to the root.
\item Let $\,d:V(T)\times V(T) \rightarrow \mathbb{N}$ be the distance defined as follows. The integer $d(v_{1}\,;\,v_{2})$ is the number of edges connecting  $v_{1}$ to $v_{2}$ if they are connected, otherwise $d(v_{1}\,;\,v_{2})=d(v_{1}\,;\,v_{3})+d(v_{3}\,;\,v_{2})$ with $v_{3}=j(v_{1}\,;\,v_{2})$. 
In Figure $1$, $d(v_{1}\,;\,r)=2\,\,$, $\,d(v_{3}\,;\,v_{4})=1$ and $d(v_{1}\,;\,v_{3})=4$.
\item A \textit{pearl tree} (or ptree) is a pair $(t,p)$ where $t$ is a planar tree and $p\in V(t)$ is called the \textit{pearl}, satisfying the property: $\forall v\in V(t)\setminus\{p\},\, d(v\,;\,p)=1$. An $S$-ptree is a pearl tree $t$ together with an $S$-labelling of $t$. 
\end{itemize}
\begin{figure}[h]
\begin{center}
\includegraphics[scale=0.3]{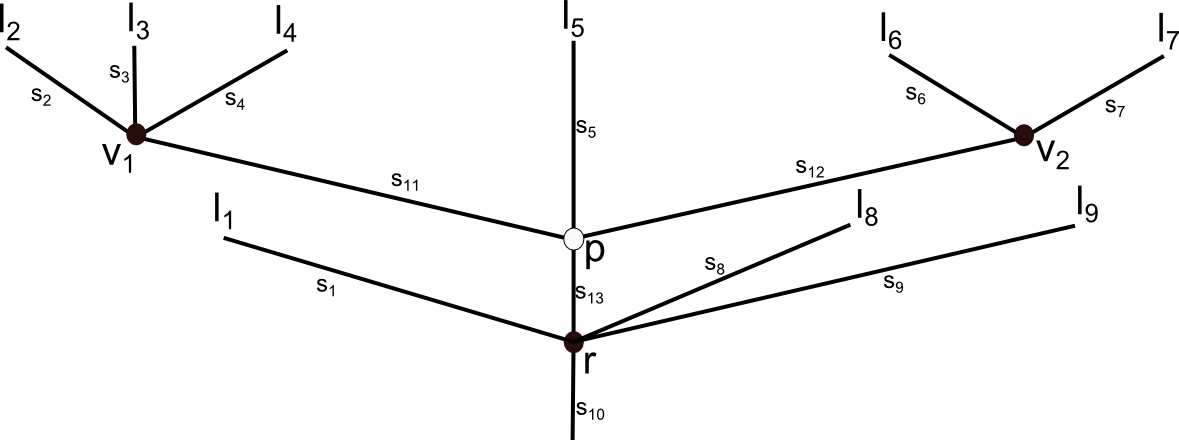}
\caption{An $S$-ptree.}
\end{center}
\end{figure}
\end{defi}

\begin{const}
The $S$-sequence $Ib_{O}(M)$ is defined by $Ib_{O}(M)(s_{1},\ldots,s_{n};s_{n+1})=$
$$
\left.\underset{\underset{f(l_{i})=s_{i}\, ,\,f(e_{0})=s_{n+1}}{(t,f,p)\in S-ptree}}{\coprod}
\left[ M\big(f(e_{1}(p)),\ldots,f(e_{|p|}(p));f(e_{0}(p))\big)\,\,\times 
\underset{v\in V(t)\setminus\{p\}}{\prod}
O\big(f(e_{1}(v)),\ldots,f(e_{|v|}(v));f(e_{0}(v))\big)\right]\right/ \sim
$$
where $\sim$ is the equivalence relation generated by
\begin{center}
\includegraphics[scale=0.4]{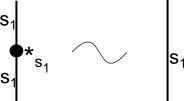}
\end{center}

Let $x$ be a point in the space $Ib_{O}(M)(s_{1},\ldots,s_{n};s_{n+1})$ indexed by an $S$-ptree $(t,f,p)$ and let $y\in O(s'_{1},\ldots, s'_{m};s_{i})$. The right infinitesimal module structure consists in grafting the $m$-corolla indexed by $y$ to the $i$-th input of $t$ and contracting the inner edge so obtained if its target does not coincide with the pearl, by using the operadic structure of $O$ as in Figure $3$:
\begin{center}
\begin{figure}[!h]
\begin{center}
\includegraphics[scale=0.41]{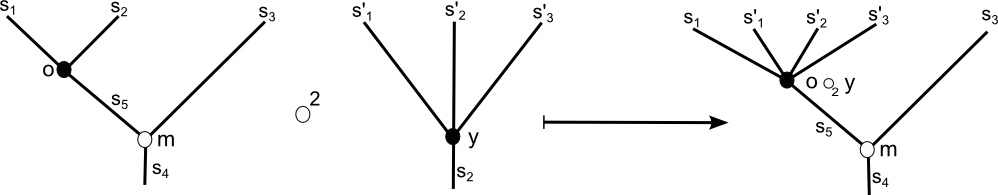}
\end{center}
\end{figure}
\end{center}

Similarly, let $x$ be a point in the space $Ib_{O}(M)(s'_{1},\ldots, s'_{m};s_{i})$ indexed by an $S$-ptree $(t,f,p)$ and let $y\in O(s_{1},\ldots,s_{n};s_{n+1})$. The left infinitesimal module structure consists in grafting the tree $t$ to the $i$-th input of the $n$-corolla indexed by $y$ and contracting the inner edge so obtained if its source does not coincide with the pearl, by using the operadic structure of $O$. These maps pass to the quotient and are continuous.

There exists an application from the $S$-sequence $M$ to  
$Ib_{O}(M)$ which maps a point $m\in M(s_{1},\ldots,s_{n};s_{n+1})$ to the pearl $n$-corolla whose leaves are labelled by $s_{1},\ldots,s_{n}$,
the trunk by $s_{n+1}$ and the pearl is indexed by $m$.\\ 
We denote by $(t,f,p,g)$ a point in $Ib_{O}(M)$ indexed by $(t,f,p)$ and labelled by $g:V(t)\rightarrow O\sqcup M$.
\end{const}

\begin{pro}\label{p1}
The functor $Ib_{O}$ is the left adjoint to the forgetful functor $$Ib_{O}(-): Coll(S)\leftrightarrows Ibimod_{O}:U.$$
\end{pro}

\begin{proof}
Given an $O$-infinitesimal bimodule $N$ and a map of $S$-sequences $h:M\rightarrow N$, we prove that there exists a unique map  
$\tilde{h}:Ib_{O}(M)\rightarrow N$ of $O$-infinitesimal bimodules such that the following diagram commutes:
$$\xymatrix{M \ar[r]^{h} \ar[d] & N \\
Ib_{O}(M) \ar[ru]_{!\, \tilde{h}} & 
}
$$
Let $(t,f,p,g)$ be a point in  $Ib_{O}(M)$. The map $\tilde{h}$ is defined by induction on $|V(t)|$ as follows. If $|V(t)|=1$, then the pearl $p$ is the only vertex and $t$ is a corolla. In this case we define $\tilde{h}((t,f,p,g))=h(g(p))$. Hence the commutativity of the previous diagram is guaranteed.\\
If $t$ has two vertices, then there exists a unique edge $e$ connecting the pearl $p$ to the other vertex $v$. There are two cases to consider:\\
- if $s(e)=p$ and $e$ is the $i$-th input of $v$ then we let $\tilde{h}((t,f,p,g))=g(v)\circ_{i} h(g(p))$.\\
- if $t(e)=p$  and $e$ is the $i$-th input of $p$ then we let  $\tilde{h}((t,f,p,g))=h(g(p)) \circ^{i} g(v)$.\\
Assume $\tilde{h}$ has been defined for $|V(t)|=n\geq 2$. Let $(t,f,p,g)\in Ib_{O}(M)$ such that $t$ has $n+1$ vertices. There exists an inner edge $e$ connecting the pearl $p$ to another vertex $v$ such that $t(e)=p$.
Let $(t',f',p,g')$ be the tree obtained by cutting off the corolla corresponding to the vertex $v$ ( $t'$ has only $n$ vertices ).  We define:  
$$
\tilde{h}((t,f,p,g))=\tilde{h}((t',f',p,g'))\circ^{i}g(v)
.$$
Due to the associativity axioms of the infinitesimal bimodule structure of $N$, $\tilde{h}$ does not depend on the choice of $v$ and $\tilde{h}$ is an infinitesimal bimodule map. The uniqueness follows from the construction.
\end{proof}

\subsection{The free bimodule}

\begin{defi}
A \textit{tree with section} (or stree) is a pair $(t,V^
{p}(t))$ where $t$ is a planar tree and $V^{p}(t)$ is a subset of $V(t)$, called the set of pearls, such that each path connecting a leaf to the trunk passes by a unique pearl and 
\begin{center}$\forall v\in V(t)\setminus V^{p}(t),\, \forall p\in V^{p}(t),\,\, j(v\,;\,p)\in \{v\,;\,p\}\Rightarrow d(v\,;\,p)=1$.\end{center} An \textit{$S$-tree} with section (or \textit{$S$-stree}) is given by a triple  $(t,V^{p}(t),f)$ such that $(t,f)$ is an $S$-tree and $(t,V^{p}(t))$ is a tree with section.

\begin{figure}[!h]
\begin{center}
\includegraphics[scale=0.42]{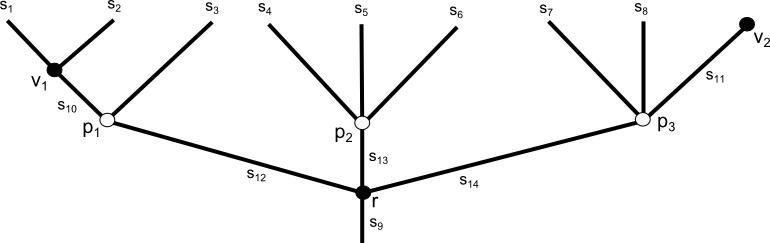}
\caption{A tree with section.}
\end{center}
\end{figure}
\end{defi}

\begin{const}
The $S$-sequence $B_{O}(M)$ is defined by $B_{O}(M)(s_{1},\ldots,s_{n};s_{n+1})=$
$$
\left.\underset{\underset{f(l_{i})=s_{i}\,;\,f(e_{0})=s_{n+1}}{(t,f,V^{p}(t))\in S-stree}}{\coprod} \left[
\underset{v\in V^{p}(t)}{\prod} M\big(f(e_{1}(v)),\ldots,f(e_{|v|}(v)); f(e_{0}(v))\big)\,\,\,
\times 
\underset{v\in V(t)\setminus V^{p}(t)}{\prod}
O\big(f(e_{1}(v)),\ldots,f(e_{|v|}(v)); f(e_{0}(v))\big)\right]\right/ \sim
$$
with $\sim$ the equivalence relation generated by
\begin{center}
\includegraphics[scale=0.4]{axiomeunite}
\end{center}

Let $x \in B_{O}(M)(s_{1},\ldots,s_{n};s_{n+1})$ indexed by a tree with section $(t,f,V^{p}(t))$ and let $y\in O(s_{1},\ldots,s_{n};s_{i})$. The right module structure consists in grafting the $m$-corolla indexed by $y$ to the $i$-th input of $t$ and contracting the inner edge so obtained if its target does not coincide with a pearl, by using the operadic structure of $O$.

Let $y$ be a point in $O(s_{1},\ldots,s_{n};s_{n+1})$ and let $\{x_{i}\}_{i=1}^{n}$ with $x_{i}\in B_{O}(M)(s^{i}_{1},\ldots,s^{i}_{n_{i}};s_{i})$ indexed by $(t_{i},f_{i},V_{i}^{p}(t))$. The left module structure consists in grafting each tree $t_{i}$ to the $i$-th input of the $n$-corolla indexed by $y$ and contracting the inner edges whose source is not a pearl by using the operadic structure of $O$, as in Figure $5$:
\begin{center}
\begin{figure}[h]
\begin{center}
\includegraphics[scale=0.41]{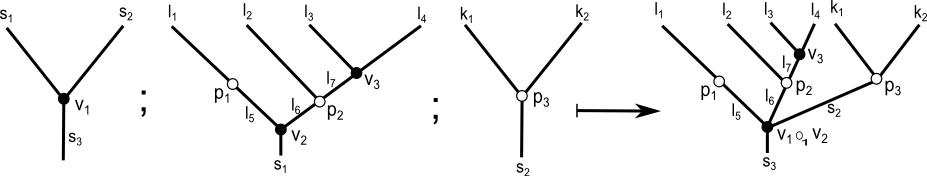}
\caption{The left module structure.}
\end{center}
\end{figure}
\end{center}

These maps pass to the quotient and are continuous. Furthermore, there exists an application from the $S$-sequence $M$ to $B_{O}(M)$ which maps a point $m\in M(s_{1},\ldots,s_{n};s_{n+1})$ to the pearl $n$-corolla whose leaves are labelled by $s_{1},\ldots,s_{n}$, the trunk by $s_{n+1}$ and the pearl is indexed by $m$.
We denote by $(t,f,V^{p}(t),g)$ a point in $B_{O}(M)$ indexed by $(t,f,V^{p}(t))$ and labelled by $g:V(t)\rightarrow O\sqcup M$.
\end{const}

\begin{pro}\label{p2}
The functor $B_{O}$ is the left adjoint to the forgetful functor
$$B_{O}(-): Coll(S)\leftrightarrows Bimod_{O}:U.$$
\end{pro}

\begin{proof}
Given an $O$-bimodule $N$ and $h:M\rightarrow N$ a map of $S$-sequences, we prove that there exists a unique map $\tilde{h}:B_{O}(M)\rightarrow N$ of $O$-bimodules such that the following diagram commutes:
$$
\xymatrix{
M \ar[r]^{h} \ar[d] & N \\
B_{O}(M) \ar[ru]_{!\, \tilde{h}} & 
}
$$
Let $(t,V^{p}(t),f,g)$ be a point in $B_{O}(M)$ and let $nb(t)$ be the cardinality of the set $V(t)\setminus V^{p}(t)$. The map $\tilde{h}$ is defined by induction on $nb(t)$. If $nb(t)=0$, then the pearl $p$ is the only vertex and $t$ is a corolla. In this case $\tilde{h}((t,V^{p}(t),f,g))=h(g(p))$.\\
If $nb(t)=1$, we denote by $v$ the unique element of $V(t)\setminus V^{p}(t)$. There are two cases to consider:\\
- if $v$ is the source of an edge $e$ which is connected to a pearl and $e$ is the $i$-th input of the unique pearl $p$, then  $$\tilde{h}((t,V^{p}(t),f,g))=h(g(p))\circ^{i}g(v).$$
- if $v$ coincides with the root, then all the pearls are connected to $v$. Let  $p_{1},\ldots,p_{k}$ be the set of ordered pearls. We define $\tilde{h}$ by 
$$\tilde{h}((t,V^{p}(t),f,g))=g(v)(h(g(p_{1})),\ldots,h(g(p_{k}))).$$
Assume $\tilde{h}$ has been defined for $nb(t)=n\geq 1$. Let $(t,V^{p}(t),f,g)\in B_{O}(M)$ such that $nb(t)=n+1$. There exists an inner edge $e$ whose target is a pearl $p_{i}$. Let $v=s(e)$ and let $(t',V^{p}(t),f',g')$ be the tree obtained from  $(t,V^{p}(t),f,g)$ by cutting off the corolla corresponding to the vertex $v$. Consequently $nb(t')=n$ and $\tilde{h}$ can be defined by induction as  
 $$\tilde{h}((t,V^{p}(t),f,g))=\tilde{h}((t',V^{p}(t),f',g'))\circ^{i} g(v).$$
Due to the associativity axioms of the bimodule structure of $N$, $\tilde{h}$ does not depend on the choice of $v$ and $\tilde{h}$ is a map of $O$-bimodules. The uniqueness follows from the construction.
\end{proof}

\section{Cofibrant replacement of the operad of monoid actions in the category of (infinitesimal) bimodules over $\sc0$}

\subsection{Model category structure on $Bimod_{O}$ and $Ibimod_{O}$}

In this section we define a model category structure on $Bimod_{O}$ and $Ibimod_{O}$ by using the previous adjunctions.  The references used for model categories are \cite{MR1361887},\cite{MR1944041}  and \cite{MR1650134}. These structures have been considered by many authors in the context of operads (symmetric, non-symmetric), algebras over operad, left-right modules over operads, most of them in the uncoloured case, see for instance Fresse \cite{MR2494775}, Berger-Moerdijk \cite{BergerMoerdijk} and Harper \cite{MR2593672}. In order to be precise, we prefer to give in details
the model category structure in our context, and take benefit of this section to state lemmas that will be useful for the sequel. 

\vspace{0.2cm}

\begin{thm}{\cite[Theorem 2.4.24]{MR1650134}}
The category  $Top$ is equipped with the following model category structure:
\begin{itemize}
\item[] \textbf{Weak equivalences} are the continuous maps $f:X\rightarrow Y$ such that $f_{0}^{\ast}:\pi_{0}(X)\rightarrow \pi_{0}(Y)$ is a bijection and $f_{n}^{\ast}:\pi_{n}(X;x)\rightarrow \pi_{n}(Y;f(x))$ is an isomorphism, $\forall x\in X$ and $\forall n>0$.

\item[] \textbf{Serre fibrations} are the continuous maps $f:X\rightarrow Y$ having the homotopy lifting property i.e., for every $CW$-complex $A$ a lift exists in every commutative diagram of the form
$$
\xymatrix{
A\times \{0\} \ar[r] \ar[d] & X \ar[d] \\
A\times [0\,,\,1] \ar[r] \ar@{-->}[ru]|\exists & Y
}
$$

\item[] \textbf{Cofibrations} are the continuous maps having the left lifting property with respect to the acyclic Serre fibrations. 
\end{itemize}
Moreover this model category is cofibrantly generated.
The cofibrations are generated by the inclusions $\partial \Delta^{n}\rightarrow \Delta^{n}$ for $n>0$, whereas the acyclic cofibrations are generated by the inclusions of the horns $\Lambda^{n}_{k}\rightarrow \Delta^{n}$ for $n>0$ and $n\geq k \geq 0$.
We call this model category the Serre model category.
\end{thm} 

\vspace{0.2cm}

\begin{cor}
The category $Coll(S)$ inherits a cofibrantly generated model category structure from the Serre model category in which a map is a cofibration, a fibration or a weak equivalence if each of its components is.
\end{cor}

\vspace{0.2cm}

\begin{lmm}{\cite{MR1944041}}\label{lemme1}
Let $A\hookrightarrow B$ be a cofibration in the Serre model category. For every space $Y$ the induced map $Top(B;Y)\rightarrow Top(A;Y)$ is a fibration.
\end{lmm}

\vspace{0.2cm}

\begin{thm}{\cite[section 2.5]{BergerMoerdijk}}\label{transfer} 
Let $\mathcal{C}_{1}$ be a cofibrantly generated model category and let $I$ (resp. $J$) be the set of generating cofibrations (resp. acyclic cofibrations). Let 
$L:\mathcal{C}_{1}\leftrightarrows \mathcal{C}_{2}:R$ be a pair of adjoint functors. Assume that $\mathcal{C}_{2}$ has small colimits and finite limits. Define a map $f$ in $\mathcal{C}_{2}$ to be a weak equivalence (resp. a fibration) if $R(f)$ is a weak equivalence (resp. fibration). If the following three conditions are satisfied:
\begin{itemize}
\item[$i)$] the functor $R$ preserves filtered colimits,
\item[$ii)$] $\mathcal{C}_{2}$ has a functorial fibrant replacement,
\item[$iii)$] for each fibrant objects $X\in \mathcal{C}_{2}$ we have a functorial path object $Path(X)$ with $X \overset{\simeq}{\rightarrow} Path(X) \twoheadrightarrow X\times X$ (a weak equivalence followed by a fibration) a factorization of the diagonal map,
\end{itemize} 
then $\mathcal{C}_{2}$ is equipped with a cofibrantly generated model category $(LI,LJ)$ with $LI=\{L(u)\, | \, u\in I\}$ and $LJ=\{L(v)\, | \, v\in J\}$. Furthermore $(L,R)$ is a Quillen pair.  
\end{thm}

\begin{app}
For the adjunction $Ib_{O}:Coll(S)\rightleftarrows Ibimod_{O}:U$. The identity induces a functorial fibrant replacement since all the objects of $Coll(S)$ are fibrants. From $M$ an $O$-infinitesimal bimodule, a functorial path object $Path(M)$ is given by the following $S$-sequence:
$$
Path(M)\big(s_{1},\ldots,s_{n};s_{n+1}\big)= Top \big( [0\,,\,1];M(s_{1},\ldots,s_{n};s_{n+1}) \big)
.$$
The $O$-infinitesimal bimodule structure and the functoriality  of $Path(-)$ are induced by that of $M$. The factorization of the diagonal map is given pointwise 
$$
\xymatrix{
M\ar[r]^{f_{1}} & Path(M) \ar[r]^{f_{2}} & M\times M
}.
$$

The application $f_{1}$ maps a point $m\in M(s_{1},\ldots,s_{n};s_{n+1})$ to the constant path in $m$. Due to the the homotopy between a path $h$ and the constant path in $h(0)$, the application $f_{1}$ is a weak equivalence.
The application $f_{2}$ maps a point $h\in Path(M)(s_{1},\ldots,s_{n};s_{n+1})$ to the pair $(h(0);h(1))\in (M\times M)(s_{1},\ldots,s_{n};s_{n+1})$. This application is a fibration since $Path(M)(s_{1},\ldots,s_{n};s_{n+1})$ is a path object in the Serre model category.

Similarly the adjunction $B_{O}(-): Coll(S)\leftrightarrows Bimod_{O}:U$ induces a cofibrantly generated model category on $Bimod_{O}$.
\end{app}

\begin{defi}
The $O$-infinitesimal bimodule $M$ is \textit{obtained from the $O$-infinitesimal bimodule $N$ by attaching cells} if $M$ is obtained as a pushout diagram of the form 
\begin{equation}
\xymatrix{
Ib_{O}(A) \ar@{^{(}->}[r]_{Ib_{O}(i)} \ar[d]_{\tilde{f}} & Ib_{O}(B) \ar[d] \\
N \ar[r] & M
}
\end{equation}
with $i$ a cofibration in $Coll(S)$, $f:A\rightarrow N$ an $S$-sequence map called the \textit{attaching map} and $\tilde{f}$ the $O$-infinitesimal bimodule map induced by $f$ (see Proposition \ref{p1}).

Similarly, an $O$-bimodule $M$ is  obtained from an $O$-bimodule $N$ by attaching cells if $M$ is obtained as a pushout diagram of the form
\begin{equation}
\xymatrix{
B_{O}(A) \ar@{^{(}->}[r]_{B_{O}(i)} \ar[d]_{\tilde{f}} & B_{O}(B) \ar[d] \\
N \ar[r] & M
}
\end{equation}
with $i$ a cofibration in $Coll(S)$, $f:A\rightarrow N$ an $S$-sequence map called the \textit{attaching map} and $\tilde{f}$ the $O$-bimodule map induced by $f$ (see Proposition \ref{p2}). In both cases the map $N\rightarrow M$ so defined is a cofibration.  
\end{defi}

\begin{defi}
Let $A$, $B$ and $C$ be three topological spaces and $f:A\rightarrow B$ be a continuous map. The space of continuous maps $g:C\rightarrow B$ such that $g_{|A}=f$ is denoted by $Top^{f}((C,A),B)$.
\end{defi}


\begin{lmm}{\cite{Tourtchine:arXiv1012.5957}}\label{lemme2}
Let $M$ and $N$ be two $O$-infinitesimal bimodules. If $M$ is obtained from $N$ by attaching cells as in $(6)$, then one has the following homeomorphism: 
$$
Ibimod_{O}^{g}((M,N);Y)\cong Coll(S)^{g\circ f}((B,A);Y),
$$
with $f$ the attaching map and $g:N\rightarrow Y$ an $O$-infinitesimal bimodule map.

Similarly, let $M$ and $N$ be two $O$-bimodules. If $M$ is obtained from $N$ by attaching cells as in $(7)$, then one has the following homeomorphism: 
$$
Bimod_{O}^{g}((M,N);Y)\cong Coll(S)^{g\circ f}((B,A);Y),
$$
with $f$ the attaching map and $g:N\rightarrow Y$ an $O$-bimodule map.
\end{lmm}




\begin{defi}\label{Remarque}
$i)\,\,$ As in \cite{Tourtchine:arXiv1012.5957} ( see also \cite[Lemma 4.24]{MR1361887} ), if $A$ and $B$ are $O$-infinitesimal bimodules (resp. $O$-bimodules), and $A^{c}$ is a cofibrant replacement of $A$ then $Ibimod_{O}(A^{c}\,;\,B)$ (resp. $Bimod_{O}(A^{c}\,;\,B)$) is independent, up to weak equivalences, of the choice of a cofibrant replacement of $A$ since every $O$-infinitesimal bimodule (resp. $O$-bimodule) $B$ is fibrant. This space is called the space of derived $O$-infinitesimal bimodule (resp. $O$-bimodule) maps from $A$ to $B$ and is denoted by:
$$
Ibimod_{O}^{h}(A\,;\,B)\,\,\,\,\,\,\,\,\,\,\big(\,\,\text{resp.}\,\,\, Bimod_{O}^{h}(A\,;\,B)\,\,\big).
$$
$ii)\,\,$ Similarly, Berger and Moerdijk define a model category structure on the category of $S$-coloured operads in \cite{BergerMoerdijk} and $Operad^{h}_{S}(A\,;\,B)$ denotes the space of derived $S$-operad maps from $A$ to $B$.\\{}\\
$iii)\,\,$ If $\mathcal{C}$ is the category $Bimod_{\sc0}$ (resp. $Operad_{\{o\,;\,c\}}$) then for any cofibrant model $A$ of $\sc0$, the family $A_{c}$ gives rise to a cofibrant replacement of $As_{>0}$ in the category $Bimod_{As_{>0}}$ (resp. $Operad$). As a consequence the homotopy fiber of the projection onto the closed part is independent (up to weak equivalences) of the choice of a cofibrant model. By abuse of notation we denote by:
\begin{align}
& p_{1}^{h}:Bimod_{\sc0}^{h}(\sc0\,;\,M)\rightarrow Bimod_{As_{>0}}^{h}(As_{>0}\,;\,M_{c})\\
& p_{2}^{h}:Operad_{\{o\,;\,c\}}^{h}(\sc0\,;\,O)\rightarrow Operad^{h}(As_{>0}\,;\,O_{c})
\end{align}
the projections onto the closed part, whenever a cofibrant model of $\sc0$ is fixed. Furthermore if the $\sc0$-bimodule $M$ and the $\{o\,;\,c\}$-operad $O$ are endowed with a map from $\tc$ then all the spaces and maps are pointed. In this case, define
\begin{center}
$\Omega\big( Bimod_{As_{>0}}^{h}(As_{>0}\,;\,M_{c}) \,;\, Bimod_{\sc0}^{h}(\sc0\,;\,M)\big)\,\,\,$ and $\,\,\,\Omega\big(  Operad^{h}(As_{>0}\,;\,O_{c}) \,;\, Operad_{\{o\,;\,c\}}^{h}(\sc0\,;\,O)\big)$ 
\end{center}
to be respectively the homotopy fiber of the projection $p_{1}^{h}$ and $p_{2}^{h}$. They are called relative loop spaces.
\end{defi}

Hence, in order to describe the spaces of derived maps and the relative loop spaces, we need to understand specific cofibrant replacement in the different categories involved. This is the aim of the two following subsections. 


\subsection{Cofibrant replacement of $\tc$ in $Ibimod_{\sc0}$} 

\begin{pro}\label{cof1}
A cofibrant replacement of the $\sc0$-infinitesimal bimodule $\tc$ is given by the $\sc0$-infinitesimal bimodule $\deltaII$:
\begin{center}
$\deltaII(n\,;\,c)=\Delta^{n}\,\,\,$ for $n\geq 0$ $\,\,\,\,\,\,\,\,$ and $\,\,\,\, \deltaII(n\,;\,o)=\Delta^{n-1}\times [0\,,\,1]\,\,\,$ for $n>0$,
\end{center}
where the structure is defined by:
$$
\begin{array}{llllllll}
\bullet & -\, \circ^{i} \ast_{2\,;\,c} & : & \deltaII(n\,;\,c)   \rightarrow \deltaII(n+1\,;\,c) & ; & (t_{1}\leq \cdots \leq t_{n}) &\mapsto (t_{1}\leq \cdots \leq t_{i} \leq t_{i} \leq \cdots \leq t_{n}), & 1\leq i \leq n,  \\ 
\bullet & -\, \circ^{i} \ast_{2\,;\,c} & : &\deltaII(n\,;\,o)   \rightarrow \deltaII(n+1\,;\,o) & ; & (t_{1}\leq \cdots \leq t_{n-1})\times t &\mapsto (t_{1}\leq \cdots \leq t_{i} \leq t_{i} \leq \cdots \leq t_{n-1})\times t, & 1\leq i\leq n-1, \\ 
\bullet & -\, \circ^{n} \ast_{2\,;\,o} & : & \deltaII(n\,;\,o)  \rightarrow \deltaII(n+1\,;\,o) &  ; & (t_{1}\leq \cdots \leq t_{n-1})\times t &\mapsto (t_{1}\leq \cdots \leq t_{n-1} \leq 1) \times t, & \\ 
\bullet & \ast_{2\,;\,c}\, \circ_{2} \, - &: & \deltaII(n\,;\,c) \rightarrow  \deltaII(n+1\,;\,c) & ; & (t_{1}\leq \cdots \leq t_{n}) &\mapsto (0\leq t_{1}\leq \cdots \leq t_{n}), &  \\ 
\bullet & \ast_{2\,;\,c}\, \circ_{1} \, - &:& \deltaII(n\,;\,c) \rightarrow  \deltaII(n+1\,;\,c) & ; & (t_{1}\leq \cdots \leq t_{n}) &\mapsto (t_{1}\leq \cdots \leq t_{n} \leq 1), &  \\ 
\bullet & \ast_{2\,;\,o}\, \circ_{2} \,- &:& \deltaII(n\,;\,o) \rightarrow  \deltaII(n+1\,;\,o) & ; & (t_{1}\leq \cdots \leq t_{n-1})\times t &\mapsto (0\leq t_{1}\leq \cdots \leq t_{n-1})\times t, &  \\ 
\bullet & \ast_{2\,;\,o}\, \circ_{1} \,- &:& \deltaII(n\,;\,c) \rightarrow  \deltaII(n+1\,;\,o) & ; & (t_{1}\leq \cdots \leq t_{n}) &\mapsto (t_{1}\leq \cdots \leq t_{n})\times 1. & 
\end{array} 
$$
\end{pro}

\begin{proof}
Since $\sc0$ is generated as a coloured operad by $\ast_{2\,;\,c}$ and $\ast_{2\,;\,o}$ with the relations $(1)$ of Definition \ref{AM}, the previous structure makes $\deltaII$ into an $\sc0$-infinitesimal bimodule. Let $\deltaII_{N}$ be the sub-$\sc0$-infinitesimal bimodule of $\deltaII$ generated by $\{\deltaII(n\,;\,c)\}_{n=0}^{N}\sqcup \{\deltaII(n\,;\,o)\}_{n=1}^{N}$ with $N\in \mathbb{N}$. By convention $\deltaII_{-1}$ is the infinitesimal bimodule $Ib_{\sc0}(\emptyset)$ and $\partial \Delta^{0} = \emptyset$. The space $\deltaII_{0}$ is obtained from $\deltaII_{-1}$ by the attaching cells:
$$\begin{array}{ccc}
Ib_{\sc0}(\partial A) & \rightarrow & Ib_{\sc0}(A) \\ 
\downarrow &  & \downarrow \\ 
\deltaII_{-1} & \rightarrow & \deltaII_{0}
\end{array}$$
with $A(0;c)=\Delta^{0}$ and the empty set otherwise.

Let  $B$ and $C$ be the $\{o\,;\,c\}$-sequences given by $B(N\,;\,o)=\Delta^{N-1}\times \{0\}$, $\,\,C(N\,;\,o)=\Delta^{N-1}\times [0\,,\,1]$, $\,\,C(N\,;\,c)=\Delta^ {N}$ and the empty set otherwise. For $N\in \mathbb{N}_{>0}$, the infinitesimal bimodule $\deltaII_{N}$ is obtained from $\deltaII_{N-1}$ by the sequence of attaching cells:
\begin{center}
$
\begin{array}{ccc}
Ib_{\sc0}(\partial B) & \rightarrow & Ib_{\sc0}(B) \\ 
\downarrow &  & \downarrow \\ 
\deltaII_{N-1} & \rightarrow & \Lambda
\end{array} \,\,\,\,\,$
and $
\,\,\,\,\,
\begin{array}{ccc}
Ib_{\sc0}(\partial C) & \rightarrow & Ib_{\sc0}(C) \\ 
\downarrow &  & \downarrow \\ 
\Lambda & \rightarrow & \deltaII_{N}
\end{array}.
$
\end{center} 
The attaching map $\partial B\rightarrow \deltaII_{N-1}$ is the restriction to the boundary of the application:
\begin{center}
 $i:\Delta^ {N-1}\rightarrow \deltaII(N\,;\,o);\,\,\, (t_{1}\leq \cdots \leq t_{N-1}) \mapsto (t_{1}\leq \cdots \leq t_{N-1})\times {0}$.
\end{center}
The homeomorphisms $\deltaII_{N-1}(N\,;\,o)\rightarrow \partial (\Delta^{N-1}\times [0\,,\,1])\setminus Int(\Delta^{N-1}\times\{0\})$ and $B(N\,;\,o)\rightarrow \Delta^{N-1}\times\{0\}$ give rise to an homeomorphism from  $\Lambda(N\,;\,o)$ to $\partial \deltaII(N\,;\,o)=\partial C(N\,;\,o)$ yielding  the right hand side attaching map. For $N\geq n$, $\deltaII_{N}(n\,;\,k)=\deltaII(n\,;\,k)$ with $k\in \{o\,;\,c\}$. Consequently, $lim_{N} \deltaII_{N}=\deltaII$ and $\deltaII$ is cofibrant. The weak equivalence between $\deltaII$ and $\tc$ is due to the convexity of $\deltaII$ in each degree.   
\end{proof}

\begin{figure}[!h]
\begin{center}
\includegraphics[scale=0.5]{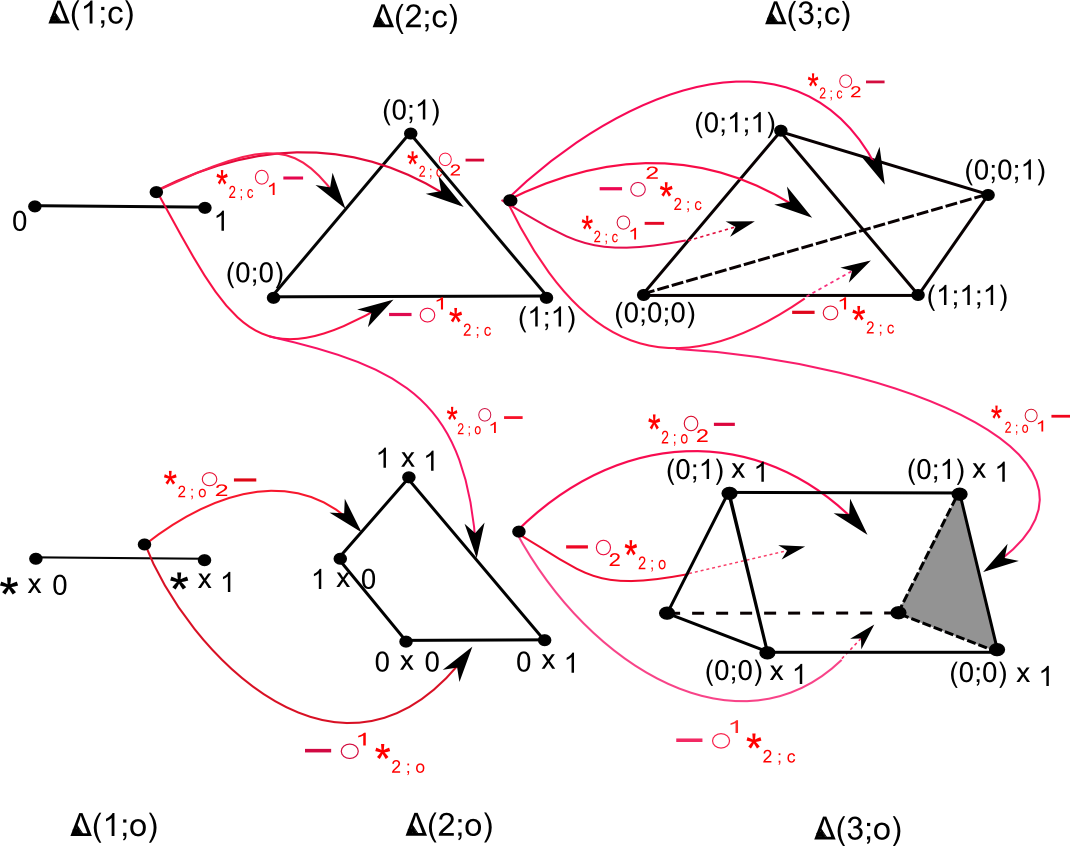}
\caption{Structure of $ \deltaII$.}
\end{center}
\end{figure}

\begin{rmk}
According to Definition \ref{Remarque}, the sequence given by $\Delta(n)=\deltaII(n\,;\,c)=\Delta^{n}$ inherits an $As_{>0}$-infinitesimal bimodule structure and it is a cofibrant replacement of $As$ in the model category $Ibimod_{As_{>0}}$ (see also \cite[Proposition 3.2]{Tourtchine:arXiv1012.5957}).
\end{rmk}

\begin{thm}\label{coro1}
Let $M$ be an $\sc0$-infinitesimal bimodule. One has:
\begin{center}
 $Ibimod^{h}_{\sc0}(\tc\,;\,M)\simeq Ibimod^{h}_{As_{>0}}(As\,;\,M_{c})\simeq sTot(M_{c})$.
\end{center}
\end{thm}

\begin{proof}
From Proposition \ref{cof1} and the previous remark, a cofibrant replacement of $\tc$ in the model category $Ibimod_{\sc0}$ is given by $\deltaII$ and a cofibrant replacement of the associative operad $As$ in the model category $Ibimod_{As_{>0}}$ is given by $\Delta$. Since $M_{c}$ is an infinitesimal bimodule over $As_{>0}$ (see Proposition \ref{SCwbimodule}), Definition \ref{Remarque} induces the following:  
\begin{center}
$Ibimod_{As_{>0}}^{h}(As\,;\,M_{c})\simeq Ibimod_{As_{>0}}(\Delta\,;\,M_{c})\,\,\,\,\,\,$ and $\,\,\,\,\,\,Ibimod_{\sc0}^{h}( \tc\,;\,M)\simeq Ibimod_{\sc0}(\deltaII\,;\,M)$.
\end{center}
Let $i$ be the inclusion defined by:
$$
\begin{array}{cccc}
i: & Ibimod_{As_{>0}}(\Delta\,;\,M_{c}) & \hookrightarrow & Ibimod_{\sc0}( \deltaII\,;\,M)
\end{array}
$$
which sends a point $f:=\left\{ f_{n\,;\,c}:\Delta^{n}\rightarrow M(n\,;\,c) \right\}_{n\in \mathbb{N}}$ to the map $g$ defined by:
$$\left\{
\begin{array}{lll}
g_{n\,;\,c}: & \Delta^{n}\rightarrow M(n\,;\,c) & ;(t_{1} \leq \cdots \leq t_{n})\mapsto f_{n\,;\,c}(t_{1} \leq \cdots \leq t_{n}) \\ 
g_{n\,;\,o}: & \Delta^{n-1}\times I\rightarrow M(n\,;\,o) & ;(t_{1} \leq \cdots \leq t_{n-1})\times t \mapsto \ast_{2\,;\,o}\circ_{1} f_{n-1\,;\,c}(t_{1} \leq \cdots \leq t_{n-1})
\end{array}\right\}.
$$
The space $Ibimod_{As_{>0}}(\Delta\,;\,M_{c})$ is a deformation retract   of $Ibimod_{\sc0}(\deltaII\,;\,M)$ with the following homotopy:
$$
\begin{array}{cccc}
H: & Ibimod_{\sc0}( \deltaII \,;\,M)\times [0\,,\,1] & \rightarrow & Ibimod_{\sc0}( \deltaII \,;\,M)
\end{array} 
$$
sending a point $(f\times u)$ to the map $H(f\,;\,u)$ given by:
\begin{center}
$\left\{
\begin{array}{lll}
H(f\,;\,u)_{n\,;\,c}: & \Delta^{n}\rightarrow M(n\,;\,c) & ;(t_{1} \leq \cdots \leq t_{n})\mapsto f_{n\,;\,c}(t_{1} \leq \cdots \leq t_{n}) \\ 
H(f\,;\,u)_{n\,;\,o}: & \Delta^{n-1}\times I\rightarrow M(n\,;\,o) & ;(t_{1} \leq \cdots \leq t_{n-1})\times t \mapsto f_{n\,;\,o}\big((t_{1} \leq \cdots \leq t_{n-1})\times (tu+(1-u))\big)
\end{array}\right\}
$.
\end{center}
The map $H$ is continuous and $H(f\,;\,1)=f$. Furthermore:
$$
\begin{array}{lllll}
H(f\,;\, 0)_{n\,;\,o}\big( (t_{1}\leq \cdots \leq t_{n-1})\times t\big) & = & f_{n\,;\,o}\big( (t_{1}\leq \cdots \leq t_{n-1})\times 1\big) & = &  f_{n\,;\,o}\big(\ast_{2\,;\,o} \circ_{1} (t_{1}\leq \cdots \leq t_{n-1})\big)\\
 & = &  \ast_{2\,;\,o} \circ_{1} f_{n-1\,;\,c}(t_{1}\leq \cdots \leq t_{n-1}) & = & \ast_{2\,;\,o} \circ_{1} f_{n-1\,;\,c}(t_{1}\leq \cdots \leq t_{n-1}).
\end{array} 
$$
So $H(f\,;\,0)$ is in the image of the inclusion map $i$ and $\forall f\in Ibimod_{As_{>0}}(\Delta\,;\,M_{c})$, $\forall u\in [0\,,\,1]$, $\,\,H(i(f)\,;\,u)=i(f)$. Indeed: 
$$
\begin{array}{lll}
\big( H(i(f)\,;\,u)\big)_{n\,;\,o}\big( (t_{1}\leq\cdots \leq t_{n-1})\times t\big) & = & i(f)_{n\,;\,o}\big( (t_{1}\leq\cdots \leq t_{n-1})\times (tu+(1-u))\big) \\
 & = & (\ast_{2\,;\,o}\circ_{1} f_{n-1\,;\,c})(t_{1}\leq\cdots \leq t_{n-1})\\
 & = & i(f)_{n\,;\,o}\big( (t_{1}\leq\cdots \leq t_{n-1})\times t\big). 
\end{array} 
$$
\end{proof}

\begin{cor}
Let $M$ be an $\sc0$-bimodule such that $M(0\,;\,c)\simeq \ast$ and let $\eta:\tc\rightarrow M$ be a map of $\sc0$-bimodules. The following weak equivalences hold: 
$$
Ibimod_{\sc0}^{h}( \tc\,;\,M) \simeq Ibimod_{As_{>0}}^{h}(As\,;\,M_{c}) \simeq \Omega Bimod^{h}_{As_{>0}}(As_{>0}\,;\,M_{c}).
$$
Similarly, let $O$ be an $\{o\,;\,c\}$-operad such that $O(0\,;\,c)\simeq O(1\,;\,c)\simeq \ast$ and let $\eta:\tc\rightarrow O$ be a map of $\{o\,;\,c\}$-operads. The following weak equivalences hold:
$$
Ibimod_{\sc0}^{h}( \tc;O) \simeq Ibimod_{As_{>0}}^{h}(As\,;\,O_{c}) \simeq \Omega^{2} Operad_{\{o\,;\,c\}}^{h}(As_{>0}\,;\,O_{c}).
$$ 
\end{cor}

\begin{proof}
It is a consequence of Theorem \ref{coro1} together with \cite[Theorem $6.2$]{Tourtchine:arXiv1012.5957} and \cite[Theorem $7.2$]{Tourtchine:arXiv1012.5957}.
\end{proof}

\subsection{Cofibrant replacement of $\sc0$ in $Bimod_{\sc0}$}

\begin{pro}\label{Rbimodule}
A cofibrant replacement of the $\sc0$-bimodule  $\sc0$ is the $\sc0$-bimodule $\car$ defined by:
\begin{center}
$\car (n\,;\,c)=[0\, ;\, 1]^{n-1}$ for $n>0$ $\,\,\,\,\,\,$ and $\,\,\,\,\,\,$ $\car (n\,;\,o)=[0\, ;\, 1]^{n-1}$ for $n>0$, 
\end{center}
whose bimodule structure is given by:
$$
\begin{array}{lllllll}
\bullet & -\,\circ^{i} \ast_{2\,;\,c} &:&\car (n\,;\,c) \rightarrow \car (n+1\,;\,c) & ; & (t_{1},\ldots , t_{n-1})\mapsto (t_{1},\ldots , t_{i-1} ,0 ,t_{i} , \ldots , t_{n-1}), \,\, 1\leq i \leq n,& \\ 
\bullet & -\,\circ^{i} \ast_{2\,;\,c} &:& \car (n\,;\,o) \rightarrow \car (n+1\,;\,o) & ; & (t_{1},\ldots , t_{n-1})\mapsto (t_{1},\ldots , t_{i-1}, 0 , t_{i} , \ldots , t_{n-1}),\,\,  1\leq i \leq n-1,& \\ 
\bullet & - \,\circ^{n} \ast_{2\,;\,o} &:& \car (n\,;\,o) \rightarrow \car (n+1\,;\,o) & ; & (t_{1},\ldots , t_{n-1})\mapsto (t_{1},\ldots , t_{n-1}, 0), &  \\ 
\bullet & \ast_{2\,;\,c}(-;-)&:&\car (n\,;\,c) \times \car (m\,;\,c) \! \rightarrow \!\car (n+m\,;\,c) & ; & (t_{1},\ldots , t_{n-1})\, ; \, (t'_{1},\ldots , t'_{m-1})\! \mapsto\! (t_{1},\ldots , t_{n-1}, 1 , t'_{1},\ldots , t'_{m-1}), &  \\ 
\bullet & \ast_{2\,;\,o}(-;-) &:& \car (n\,;\,c) \times \car (m\,;\,o)\! \rightarrow\! \car (n+m\,;\,o) & ; & (t_{1},\ldots , t_{n-1})\, ; \, (t'_{1},\ldots , t'_{m-1})\! \mapsto\! (t_{1},\ldots , t_{n-1}, 1 , t'_{1},\ldots , t'_{m-1}). & \end{array} 
$$
\end{pro}

\begin{proof}
Since $\sc0$ is generated as a coloured operad by $\ast_{2\,;\,c}$ and $\ast_{2\,;\,o}$ with the relations $(1)$ of Definition \ref{AM}, the previous structure induces an $\sc0$-bimodule structure on $\square$. For $N>0$ let $\car_{N}$ be the sub-$\sc0$-bimodule of $\car$ generated by $\{\car(n\,;\,k)\}_{n\in \{1,\ldots,N\}}^{k\in \{o\,;\,c\}}$. By convention $\car_{0}$ is the $\sc0$-bimodule $B_{\sc0}(\emptyset)$. The bimodule $\car_{N}$ is obtained from $\car_{N-1}$ by the attaching cells:
$$
\begin{array}{ccc}
B_{\sc0}(\partial A) & \rightarrow & B_{\sc0}(A) \\ 
\downarrow &  & \downarrow \\ 
\car_{N-1} & \rightarrow & \car_{N}
\end{array} 
$$ 
with $A$ the $\{o;c\}$-sequence defined by $A(N\,;\,c)=A(N\,;\,o)=[0\, ;\, 1]^{N-1}$ and the empty set otherwise.\\
For  $N\geq n$, $\car_{N}(n\,;\,k)=\car(n\,;\,k)$ with $k\in \{o\,;\,c\}$. Consequently, $lim_{N} \car_{N} = \car$ and $\car$ is cofibrant. The weak equivalence between $\car$ and $\sc0$ is due to the convexity of $\car$ in each degree.
\end{proof}

\begin{rmk}
According to Definition \ref{Remarque}, the sequence given by $\square_{c}(n)=\square(n\,;\,c)$ inherits an $As_{>0}$-bimodule structure and it is a cofibrant replacement of $As_{>0}$ in the model category $Bimod_{As_{>0}}$ (see \cite[Proposition 4.1]{Tourtchine:arXiv1012.5957}).
\end{rmk}

\section{Relative delooping of $sTot(M_{o})$}
Let $M$ be an $\sc0$-bimodule endowed with a map $\eta: \tc\rightarrow M$. 
Since the semi-cosimplicial space $M_{o}$ is not a monoid in $(Top^{\Delta_{inj}}\,,\,\boxtimes)$ (see Proposition \ref{definMS}), $M_{o}$  is not a bimodule over $As_{>0}$ and we can not expect that its semi-totalization has the homotopy type of a loop space. However, we will use the left module structure on $M_{o}$ to prove that the pair  $(sTot(M_{c})\,;\,sTot(M_{o}))$ has the homotopy type of an $\mathcal{SC}_{1}$-space. The first step consists in showing that $sTot(M_{o})$ is weakly equivalent to the homotopy fiber of the map $(8)$ of Definition \ref{Remarque}. The next definition gives a model of this homotopy fiber using the cofibrant replacement $\square$ of $\sc0$ (see Proposition  \ref{Rbimodule}).

\begin{defi}\label{Defifi}
Let $\eta:\tc \rightarrow M$ be an $\sc0$-bimodule map and let $\square \times I$ be the $\{o\,;\,c\}$-sequence defined by:
\begin{center}
$(\square  \times I)(n\,;\,c)=\square(n\,;\,c)\times [0\,,\,1]\,\,\,\,\,\,$ and $\,\,\,\,\,\,(\square  \times I)(n\,;\,o)=\square(n\,;\,o)\times \{1\},\,\,$ for $n >0$.
\end{center}
A \textit{relative loop} in $M$ is an $\{o\,;\,c\}$-sequence map $g$ from $\square \times I$ to $M$ defined by:
\begin{center}
$g_{n\,;\,c}:\square(n\,;\,c)\times [0\,,\,1]\rightarrow M(n\,;\,c)\,\,$  and  $\,\,g_{n\,;\,o}:\square(n\,;\,o)\times \{1\} \rightarrow M(n\,;\,o)$, for $n>0$, 
\end{center}
satisfying:
$$
\begin{array}{lll}
\bullet & g_{n\,;\,c}(x\circ^{i} \ast_{2\,;\,c}\,;\,t)=g_{n-1\,;\,c}(x\,;\,t)\circ^{i} \ast_{2\,;\,c} & \text{for}\,\, x\in \square(n-1\,;\,c) \,\,\text{and}\,\, 1\leq i \leq n-1,\\ 
\bullet & g_{n\,;\,c}\big( \ast_{2\,;\,c}(x\,;\,y)\,;\,t\big)=\ast_{2\,;\,c}\big( g_{l\,;\,c}(x\,;\,t)\,;\,g_{n-l\,;\,c}(y\,;\,t)\big) & \text{for}\,\, x\in \square(l\,;\,c) \,\, \text{and} \,\, y\in \square(n-l\,;\,c),\\
\bullet & g_{n\,;\,o}(x\circ^{i} \ast_{2\,;\,c}\,;\,1)=g_{n-1\,;\,o}(x\,;\,1)\circ^{i} \ast_{2\,;\,c} & \text{for} \,\, x\in \square(n-1\,;\,o) \,\, \text{and}\,\, 1\leq i \leq n-2,\\
\bullet & g_{n\,;\,o}(x\circ^{n-1} \ast_{2\,;\,o}\,;\,1)=g_{n-1\,;\,o}(x\,;\,1)\circ^{n-1} \ast_{2\,;\,o} & \text{for}\,\, x\in \square(n-1\,;\,o),\\
\bullet & g_{n\,;\,o}\big( \ast_{2\,;\,o}(x\,;\,y)\,;\,1\big) =\ast_{2\,;\,o}\big( g_{l\,;\,c}(x\,;\,1)\,;\,g_{n-l\,;\,o}(y\,;\,1)\big) & \text{for}\,\, x\in \square(l\,;\,c) \,\,\text{and}\,\, y\in \square(n-l;o)
\end{array} 
$$
with the boundary conditions: $g_{n\,;\,c}(x\,;\,0)=\eta(\ast_{n\,;\,c})$ for $x\in \square(n\,;\,c)$.\\
This model for the space of relative loops is denoted by $\Omega \big( Bimod_{As_{>0}}(\square_{c} \,;\, M_{c})\,;\, Bimod_{\sc0}(\square \,;\, M)\big)$.
\end{defi}

\begin{thm}\label{mainthm}
If $M$ is an $\sc0$-bimodule endowed with  a map of $\sc0$-bimodules $\eta:\tc\rightarrow M$ then
$$
sTot(M_{o})\simeq \Omega \big( Bimod^{h}_{As_{>0}}(As_{>0} \,;\, M_{c})\,;\, Bimod_{\sc0}(\sc0 \,;\, M)\big).
$$
\end{thm}

\begin{proof}
It is a consequence of Proposition \ref{Lmm1TF} and Proposition \ref{Lmm2TF}.
\end{proof}

\begin{notat}
Let $M$ be an $\sc0$-bimodule endowed with a map $\eta:\tc\rightarrow M$. The $\{o\,;\,c\}$-sequence $M^{\ast}$ given by
\vspace{-0.3cm} 
\begin{center}
$M^{\ast}(n\,;\,c)=\eta(\ast_{n\,;\,c})\,\,\,$ for $\,\,n\geq 0$, $\,\,\,\,\,\,\,\,M^{\ast}(n\,;\,o)=M(n\,;\,o)\,\,\,$ for $n>0$
\end{center}
and the empty set otherwise, inherits from $M$ an $\sc0$-bimodule structure  with a map $\eta:\tc\rightarrow M^{\ast}$. 
\end{notat}

\begin{pro}\label{Lmm1TF}
The space $sTot(M_{o})$ is weakly equivalent to $Bimod_{\sc0}^{h}(\sc0\,;\, M^{\ast})$. 
\end{pro}

\begin{proof}
As seen in the first section $sTot(M_{o})\simeq Ibimod^{h}_{As_{>0}}(As\,;\,M_{o})$ using the structure $(2)$. The first step of the proof consists in building a cofibrant replacement $\tilde{\square}$ of $As$ in the category of infinitesimal bimodule over $As_{>0}$ so that there exists a map $\xi: Bimod_{\sc0}(\square\,;\,M^{\ast})\rightarrow Ibimod_{As_{>0}}(\tilde{\square}\,;\,M_{o})$.   
Let us recall that a point $g\in Bimod_{\sc0}(\square \,;\, M^{\ast})$ is described by:
\begin{center}
$
\left\{
\begin{array}{llll}
g_{n\,;\,c}: & \square(n\,;\,c)\rightarrow M^{\ast}(n\,;\,c); & x\mapsto \eta(\ast_{n\,;\,c}), & $ for $ n>0, \\ 
g_{n\,;\,o}: & \square(n\,;\,o)\rightarrow M^{\ast}(n\,;\,o), &  & $ for $ n>0.
\end{array} 
\right.
$
\end{center}
satisfying:
$$
\begin{array}{lllllll}
\bullet & g_{n\,;\,o}(x\, \circ^{i}\,\ast_{2\,;\,c}) & = & g_{n-1\,;\,o}(x)\, \circ^{i}\,\ast_{2\,;\,c}\, , & & &  \text{for}\,\,  x\in\square(n-1\,;\,o) \,\, \text{and}\,\,  i\neq n-1, \\ 
\bullet & g_{n\,;\,o}(x\, \circ^{n-1}\,\ast_{2\,;\,o}) & = & g_{n-1\,;\,o}(x)\, \circ^{n-1}\,\ast_{2\,;\,o}\, , & & &  \text{for}\,\, x\in\square(n-1\,;\,o), \\ 
\bullet & g_{n\,;\,o}\big( \ast_{2\,;\,o}(x\,;\,y) \big) & = & \ast_{2\,;\,o}\big( g_{l\,;\,c}(x)\,;\, g_{n-l\,;\,o}(y)\big) & = & \ast_{2\,;\,o}\big( \eta(\ast_{l\,;\,c})\,;\, g_{n-l\,;\,o}(y)\big)\, , &  \text{for}\,\, x\in\square(l\,;\,c) \,\,\text{and}\,\, y\in \square(n-l\,;\,o).
\end{array} 
$$
Define $\sim$ to be the equivalence relation on $[0\,,\,1]^{n}$ generated by:
\begin{center}
$(t_{1},\ldots,t_{n})\sim (t'_{1},\ldots,t'_{n})\,\, \Leftrightarrow\,\,\exists i\in \{1,\ldots,n\}$ such that 
$
\left\{ \begin{array}{ll}
\bullet\, t_{i}=t'_{i}=1 &  \\ 
\bullet\, t_{j}=t'_{j} & $ for $ j>i
\end{array}\right. 
$.
\end{center}
We denote by $\tilde{\square}$ the sequence $\{ \tilde{\square}(n)=[0\,,\,1]^{n}/\sim\}_{n\geq 0}$.\\ The map $g$ induces a sequence map $\tilde{g}:=\{\tilde{g}_{n+1}:\tilde{\square}(n)\rightarrow M^{\ast}(n+1\,;\,o)=M^{n}_{o}\}_{n\geq 0}$. Indeed if $(t_{1},\ldots,t_{n})\sim (t'_{1},\ldots,t'_{n})$ then there exists $i$ such that $t_{i}=t'_{i}=1$ and $t_{j}=t'_{j}$ for $j>i$. So the following equalities hold:
\begin{center}$
\begin{array}{lllll}
g_{n+1\,;\,o}(t_{1},\ldots,t_{n}) & = & g_{n+1\,;\,o}(t_{1},\ldots,t_{i-1},1,t_{i+1},\ldots, t_{n}) & = & g_{n+1\,;\,o}\big( \ast_{2\,;\,o}\,\big( (t_{1},\ldots,t_{i-1}) \,;\, (1,t_{i+1},\ldots, t_{n})\big) \big) \\ 
 & = & \ast_{2\,;\,o}\,\big( g_{i\,;\,c}(t_{1},\ldots,t_{i-1}) \,;\, g_{n-i\,;\,o}(t_{i+1},\ldots, t_{n})\big) & = & g_{n+1\,;\,o}(t'_{1},\ldots,t'_{n}) 
\end{array} 
$\end{center}

Let us prove that $\tilde{\square}$ is a cofibrant replacement of $As$ as an $As_{>0}$-infinitesimal bimodule. The infinitesimal bimodule structure is given by:
$$
\begin{array}{lllllll}
i)&-\,\circ^{i}\,\ast_{2} & : & \tilde{\square}(n)\rightarrow \tilde{\square}(n+1) & ; & (t_{1},\ldots,t_{n})\mapsto (t_{1},\ldots,t_{i-1},0,t_{i+1},\ldots,t_{n}), & \text{for}\,\, 1\leq i \leq n,  \\ 
ii)&\ast_{2}\,\circ_{1}\,- & : & \tilde{\square}(n)\rightarrow \tilde{\square}(n+1) & ; & (t_{1},\ldots,t_{n})\mapsto (t_{1},\ldots,t_{n},0), & \\ 
iii)&\ast_{2}\,\circ_{2}\,- & : & \tilde{\square}(n)\rightarrow \tilde{\square}(n+1) & ; & (t_{1},\ldots,t_{n})\mapsto (1,t_{1},\ldots,t_{n}). & 
\end{array} 
$$
This structure satisfies the infinitesimal bimodule axioms over $As_{>0}$ and it makes $\tilde{g}$ into an $As_{>0}$-infinitesimal bimodule map. Furthermore $\tilde{\square}$ is a cofibrant replacement of the $As_{>0}$-infinitesimal bimodule $As$:\\{}\\
\textbf{Cofibrant:} let $\tilde{\square}_{n}$ be the sub-$As_{>0}$-infinitesimal bimodule of $\tilde{\square}$ generated by $\{\tilde{\square}(i)\}_{i=0}^{n}$ for $n\in \mathbb{N}$. By convention  $\tilde{\square}_{-1}$ is the $As_{>0}$-infinitesimal bimodule $Ib_{As_{>0}}(\emptyset)$. Let us notice that the boundary of $\tilde{\square}(n)$ is determined by $\tilde{\square}(n-1)$ and its infinitesimal bimodule structure. Indeed  the map $[0\,,\,1]^{n}\rightarrow \tilde{\square}(n)$ preserves the boundary and by  definition a point in $\partial [0\,,\,1]^{n}$ has one of the following form:
\begin{center}
$(t_{1},\ldots,t_{l-1},0,t_{l+1},\ldots,t_{n})\,\,\,\,\,$ or $\,\,\,\,\,(t_{1},\ldots,t_{l-1},1,t_{l+1},\ldots,t_{n})$.
\end{center}
In the first case, the class of such a point lies in $\tilde{\square}_{n-1}$ by the axioms $(i)$ and $(ii)$. In the second case we have the  identification:
\begin{center}
$[(t_{1},\ldots,t_{l-1},1,t_{l+1},\ldots,t_{n})]=[(\underset{l}{\underbrace{1,\ldots,1}},t_{l+1},\ldots,t_{n})] = \ast_{2}\, \circ_{2}\, [(\underset{l-1}{\underbrace{1,\ldots,1}},t_{l+1},\ldots,t_{n})]$
\end{center}
Consequently $\tilde{\square}_{n}$ is obtained from $\tilde{\square}_{n-1}$ by the pushout diagram:
\begin{equation}
\xymatrix{
Ib_{As_{>0}}(\partial A) \ar[r] \ar[d]_{\tilde{q}} & Ib_{As_{>0}}(A) \ar[d] \\
\tilde{\square}_{n-1} \ar[r] & \tilde{\square}_{n}
}
\end{equation}
where $A$ is the sequence given by $A(n)=[0\,,\,1]^{n}$ and the empty set otherwise. The attaching map is the restriction of the quotient map $q:[0\,,\,1]^{n}\rightarrow [0\,,\,1]^{n}/\sim$ to the boundary. Moreover if $n\geq i$ then $\tilde{\square}_{n}(i)=\tilde{\square}(i)$ and the map $\partial A\rightarrow A$ is a cofibration. So $lim_{n}\,\tilde{\square}_{n}=\tilde{\square}$ and $\tilde{\square}$ is cofibrant.

This construction implies that $\tilde{\square}(m)$ is a $CW$-complex. Let us recall that if $A(n)=[0\,,\,1]^{n}$ and the empty set otherwise, then the points in $Ib_{As_{>0}}(A)(m)$ are the pairs $(t\,;\,x)$ with $x\in A(n)$ and $t$ a $\{c\}$-ptree satisfying:
\begin{equation}
\bullet \,\,\,t\,\,\text{has}\,\,m\,\,\text{leaves}\,\,\,\,\,\,\,\,\,\bullet \,\,\,\forall v\in V(t)\setminus \{p\},\,|v|>1 \,\,\,\,\,\,\,\,\,\bullet \,\,\, |p|=n.
\end{equation} 
We denote by $tr_{m}^{n}$ the number of $\{c\}$-ptrees satisfying Relation $(11)$. The space $\tilde{\square}_{0}(m)$ is the disjoint union of $tr_{m}^{0}$ points, that is, a $CW$-complex. Assume $\tilde{\square}_{n-1}(m)$ is a $CW$-complex  $\forall m\geq 0$.
For $m\leq n-1$, $\tilde{\square}_{n}(m)=\tilde{\square}_{n-1}(m)=\tilde{\square}(m)$ is a $CW$-complex.
The pushout $(10)$ implies that $\tilde{\square}_{n}(n)=\tilde{\square}(n)$ is a $CW$-complex. Finally, for $m>n$, the space $\tilde{\square}_{n}(m)$ is obtained from the $CW$-complex $\tilde{\square}_{n-1}(m)$ by attaching $tr_{m}^{n}$ cells of dimension $n$ according to the infinitesimal bimodule structure over $As_{>0}$, thus is a $CW$-complex.\\{}\\
\textbf{Contractible:} The map $q:[0\,,\,1]^{n}\rightarrow \tilde{\square}(n)$ is a continuous map between compact $CW$-complexes. Since the fiber of $q$ over a point $(t_{1},\ldots, t_{i-1},1,t_{i+1},\ldots,t_{n})$, with $t_{j}\neq 1$ for $j>i$, is homeomorphic to the contractible space $[0\,,\,1]^{i-1}$, the map $q$ is a weak equivalence  \cite[Main Theorem]{Smale}. Hence $\tilde{\square}(n)$ is contractible.\\


Since $\tilde{\square}$ is a cofibrant replacement of $As$ as an infinitesimal bimodule over $As_{>0}$, the semi-totalization $sTot(M_{o})$ is weakly equivalent to $Ibimod_{As_{>0}}(\tilde{\square}\,;\,M_{o})$ and we have a map:
\begin{center}
$
\xi: Bimod_{Act_{>0}}(\square\,;\,M^{\ast})\rightarrow Ibimod_{As_{>0}}(\tilde{\square}\,;\,M_{o})\,\,;\,\, g\mapsto \tilde{g}.
$
\end{center}
In order to prove that $\xi$ is a weak equivalence, we will introduce two towers of fibrations.
For $k\geq 0$, define $A_{k}$ and $B_{k}$ to be the subspaces :
$$A_{k}\subset \underset{\text{reduced to a point}}{\underbrace{\underset{i=1}{\overset{k+1}{\prod}}\, Top\big( \square(i\,;\,c)\,;\, M^{\ast}(i\,;\,c)\big)}}\times \underset{i=1}{\overset{k+1}{\prod}}\, Top\big( \square(i\,;\,o)\,;\,M^{\ast}(i\,;\,o)\big)\,\,\,\,\,\, \text{and} \,\,\,\,\,B_{k}\subset \underset{i=0}{\overset{k}{\prod}}\, Top\big( \tilde{\square}(i)\,;\, M^{i}_{o}\big)$$
with $A_{k}$ satisfying the $\sc0$-bimodule relations and $B_{k}$ the $As_{>0}$-infinitesimal bimodule relations. In other words $A_{k}$ and $B_{k}$ are respectively the spaces $Bimod_{\sc0}(\square_{k+1}\,;\,M^{\ast})$ and $Ibimod_{As_{>0}}(\tilde{\square}_{k}\,;\,M_{o})$ where $\square_{k+1}$ is the sub-$\sc0$-bimodule introduced in the proof of Proposition \ref{Rbimodule}. The projection:
$$
\underset{i=1}{\overset{k+1}{\prod}}\, Top\big( \tilde{\square}(i)\,;\,M_{o}^{i}\big)\,\rightarrow \, \underset{i=1}{\overset{k}{\prod}}\, Top\big( \tilde{\square}(i)\,;\,M_{o}^{i}\big)
$$
induces a map $B_{k+1}\rightarrow B_{k}$. From Lemma \ref{lemme1}, the following map is a fibration:
\begin{center}
$
Top\big( \tilde{\square}(k+1)\,;\,M_{o}^{k+1}\big)\,\rightarrow \,  Top\big( \partial\tilde{\square}(k+1)\,;\,M_{o}^{k+1}\big)
$
\end{center}
The space $B_{k+1}$ is obtained from $B_{k}$ by the pullback diagram:
\begin{center}
$
\xymatrix@R=0.5cm{
B_{k+1} \ar[r] \ar[d] & Top\big( \tilde{\square}(k+1)\,;\,M_{o}^{k+1}\big) \ar[d] \\
B_{k} \ar[r] & Top\big( \partial\tilde{\square}(k+1)\,;\,M_{o}^{k+1}\big)
}
$
\end{center}
Since the fibrations are preserved by pullbacks, $B_{k+1}\rightarrow B_{k}$  is a fibration. Similarly the next pullback square makes the map $A_{k+1}\rightarrow A_{k}$ induced by the projection into a fibration:
\begin{center}
$
\xymatrix@R=0.5cm{
A_{k+1} \ar[d] \ar[r] & Top\big( \square(k+2\,;\,c)\,;\,M^{\ast}(k+2\,;\,c)\big)\times Top\big( \square(k+2\,;\,o)\,;\,M^{\ast}(k+2\,;\,o)\big) \ar[d]\\
A_{k} \ar[r] & Top\big( \partial\square(k+2\,;\,c)\,;\,M^{\ast}(k+2\,;\,c)\big)\times Top\big( \partial\square(k+2\,;\,o)\,;\,M^{\ast}(k+2\,;\,o)\big)
}
$
\end{center}
So we consider the two towers of fibrations:
\begin{center}
$
\xymatrix@R=0.1cm{
A_{0} & A_{1} \ar[l] & \cdots \ar[l] & A_{k} \ar[l] & A_{k+1} \ar[l] & \cdots \ar[l] \\
B_{0} & B_{1} \ar[l] & \cdots \ar[l] & B_{k} \ar[l] & B_{k+1} \ar[l] & \cdots \ar[l]
}
$
\end{center}
so that:\\{}\\
$
\begin{array}{lllllll}
A_{\infty} & = & lim_{k}\, A_{k} & \simeq & holim_{k}\, A_{k} & \simeq & Bimod_{\sc0}(\square\,;\,M^{\ast}), \\ 
B_{\infty} & = & lim_{k}\, B_{k} & \simeq & holim_{k}\, B_{k} & \simeq & Ibimod_{As_{>0}}(\tilde{\square}\,;\,M_{o}).
\end{array} 
$\\{}\\
By restriction, the map $\xi$ induces an application between the two towers:
\begin{center}
$
\xymatrix@R=0.5cm{
A_{0} \ar[d]^{\xi_{0}} & A_{1} \ar[l] \ar[d]^{\xi_{1}} & \cdots \ar[l] & A_{k} \ar[l] \ar[d]^{\xi_{k}} & A_{k+1} \ar[l] \ar[d]^{\xi_{k+1}}& \cdots \ar[l] \\
B_{0} & B_{1} \ar[l] & \cdots \ar[l] & B_{k} \ar[l] & B_{k+1} \ar[l] & \cdots \ar[l]
}
$
\end{center}
with $\xi\,=\, lim_{k}\,\xi_{k}\, = \, holim_{k}\, \xi_{k}$. Consequently, $\xi$ is a weak equivalence if each $\xi_{k}$ is a weak equivalence. We will prove this result by induction on $k$:\\
$\bullet\,\,\,\xi_{0}$ and $\xi_{1}$ coincide with the identity. They are weak equivalences.\\
$\bullet\,\,$ Assume that $\xi_{k-1}$ is a weak equivalence. We consider the following diagram where $g$ is a point in $A_{k-1}$, $F_{A}$ is the fiber over $g$ and $F_{B}$ the fiber over $\xi_{k-1}(g)$. Since the two left horizontal arrows are fibrations, the map $\xi_{k}$ is a weak equivalence if the induced map $\xi_{g}$ is a weak equivalence.
$$
\xymatrix@R=0.6cm{
A_{k-1} \ar[d]^{\xi_{k-1}} & A_{k} \ar[l] \ar[d]^{\xi_{k}} & F_{A} \ar[l] \ar[d]^{\xi_{g}} \\
B_{k-1} & B_{k} \ar[l] & F_{B} \ar[l]
}
$$
From Lemma \ref{lemme2} the fiber $F_{A}$ is homeomorphic to the space $Top^{g_{k+1\,;\,o}}\big(\,\big( [0\,,\,1]^{k} \,;\,\partial [0\,,\,1]^{k}\big)\,;\, M(k+1\,;\,o)\big)$. Similarly $\tilde{\square}_{k}$ is obtained from $\tilde{\square}_{k-1}$ by the pushout diagram $(10)$. So the fiber  $F_{B}$ is homeomorpic to the space $Top^{\xi_{k-1}(g)_{k}\circ q}\big(\,\big( [0\,,\,1]^{k} \,;\,\partial [0\,,\,1]^{k}\big)\,;\,M(k+1\,;\,o)\big)$ and we have the commutative square:
$$
\xymatrix@R=0.6cm{
F_{A} \ar[r] \ar[d]_{\xi_{g}} &  Top^{g_{k+1\,;\,o}}\big(\,\big( [0\,,\,1]^{k} \,;\,\partial [0\,,\,1]^{k}\big)\,;\, M(k+1\,;\,o)\big) \ar@{=}[d]^{id} \\
F_{B} \ar[r] & Top^{\xi_{k-1}(g)_{k}\circ q}\big(\,\big( [0\,,\,1]^{k} \,;\,\partial [0\,,\,1]^{k}\big)\,;\,M(k+1\,;\,o)\big)
}
$$
Consequently $\xi_{k}$ is a weak equivalence.

\end{proof}

\begin{pro}\label{Lmm2TF}
The space $\Omega \big( Bimod^{h}_{As_{>0}}(As_{>0} \,;\, M_{c})\,;\, Bimod^{h}_{\sc0}(\sc0 \,;\, M)\big)$ is weakly equivalent to the space    
$Bimod^{h}_{\sc0}(\sc0\,;\, M^{\ast})$.
\end{pro}

\begin{proof}
In this proof $\square$ will serve as a cofibrant model of the $\sc0$-bimodule $\sc0$. We can consider $Bimod_{\sc0}(\square\,;\, M^{\ast})$ as a subspace of $\Omega \big( Bimod_{As_{>0}}(\square_{c} \,;\, M_{c})\,;\, Bimod_{\sc0}(\square \,;\, M)\big)$ through the inclusion:
$$
\begin{array}{cccc}
i\,: & Bimod_{\sc0}(\square\,;\, M^{\ast}) & \rightarrow & \Omega \big( Bimod_{As_{>0}}(\square_{c} \,;\, M_{c})\,;\, Bimod_{\sc0}(\square \,;\, M)\big) \\ 
 & g & \mapsto & 
 \left\{\begin{array}{cccc}
 \tilde{g}_{n\,;\,c}\,: & \square(n\,;\,c)\times [0\,,\,1] \rightarrow M(n\,;\,c) & ; & (x\,;\,t)\mapsto \eta(\ast_{n\,;\,c}) \\ 
 \tilde{g}_{n\,;\,o}\,: & \square(n\,;\,o)\times \{1\} \rightarrow M(n\,;\,o)  & ; &  (x\,;\,1)\mapsto g_{n\,;\,o}(x)
 \end{array}\right. 
\end{array} 
$$
In order to show that $i$ is a weak equivalence, we introduce two towers of fibrations. One of them is the tower $A_{k}$ of Proposition \ref{Lmm1TF}. The second one is defined by:
$$
C_{k}\subset \underset{i=1}{\overset{k+1}{\prod}}\, Top\big( \square(i\,;\,c)\times [0\,,\,1]\,;\, M(i\,;\,c)\big)\times \underset{i=1}{\overset{k+1}{\prod}}\, Top\big( \square(i\,;\,o)\,;\,M(i\,;\,o)\big)
$$
satisfying the relations of Definition \ref{Defifi}. The map $C_{k+1}\rightarrow C_{k}$ induced by the projection is a fibration due to Lemma \ref{lemme1} and the following pullback diagram:


$$
\xymatrix{
C_{k+1} \ar[r] \ar[d] & Top\big( \square(k+2\,;\,c)\times [0\,,\,1]\,;\,M(k+2\,;\,c)\big) 
\times Top\big(\square(k+2\,;\,o)\,;\,M(k+2\,;\,o)\big) \ar[d] \\
C_{k} \ar[r] & Top\big(\partial'(\square(k+2\,;\,c)\times [0\,,\,1]) \,;\,M(k+2\,;\,c)\big) \times Top\big(\partial\square(k+2\,;\,o)\,;\,M(k+2\,;\,o)\big)
}
$$
where 
$$\partial'(\square(k+2\,;\,c)\times [0\,,\,1])=\square(k+2\,;\,c)\times \{0\} \cup \partial \square(k+2\,;\,c)
\times [0\,,\,1].$$
The restriction of the inclusion $i$ induces a map between the two towers:
\begin{center}
$
\xymatrix{
A_{0} \ar[d]^{i_{0}} & A_{1} \ar[l] \ar[d]^{i_{1}} & \cdots \ar[l] & A_{k} \ar[l] \ar[d]^{i_{k}} & A_{k+1} \ar[l] \ar[d]^{i_{k+1}}& \cdots \ar[l] \\
C_{0} & C_{1} \ar[l] & \cdots \ar[l] & C_{k} \ar[l] & C_{k+1} \ar[l] & \cdots \ar[l]
}
$
\end{center}
We will prove that $i$ is a weak equivalence by induction on $k$:\\
$\bullet\,\,$ If $k=0$, a point in $C_{0}$ is a pair $(g_{1\,;\,c}\,;\,g_{1\,;\,o})$ and the points in the image of $i_{0}$ are the pairs satisfying:
$$
g_{1\,;\,c}:\square(1\,;\,c)\times [0\,,\,1]\rightarrow M(1\,;\,c)\,\,;\,\, (\ast\,;\,t)\mapsto \eta(\ast_{1\,;\,c})
$$
Since $g_{1\,;\,c}(\ast\,;\,0)=\eta(\ast_{1\,;\,c})$ for any pair in $C_{0}$, the inclusion $i_{0}$ induces the following deformation retract:
\begin{center}
$
\begin{array}{cccl}
H\,: & C_{0}\times[0\,,\,1] & \rightarrow & C_{0} \\ 
 & y=\big( (g_{1\,;\,c}\,;\,g_{1\,;\,o})\,;\,t_{1}\big) & \mapsto & 
 \left\{
 \begin{array}{l}
 H(y)_{1\,;\,c}(\ast\,;\,t)=g_{1\,;\,c}\big(\ast\,;\,t(1-t_{1})\big) \\ 
 H(y)_{1\,;\,o}(\ast\,;\,1)=g_{1\,;\,o}(\ast\,;\,1)
 \end{array} 
 \right.
\end{array} 
$
\end{center}
$\bullet\,\,$ From now on we assume that $i_{k-1}$ is a weak equivalence for $k\geq 1$. We consider the following diagram where $g$ is a point in $A_{k-1}$, $F_{A}$ is the fiber over $g$ and $F_{C}$ the fiber over $i_{k-1}(g)$. Since the two left horizontal arrows are fibrations, the map $i_{k}$ is a weak equivalence if the induced map $i_{g}$ is a weak equivalence. 
$$
\xymatrix{
A_{k-1} \ar[d]^{i_{k-1}} & A_{k} \ar[l] \ar[d]^{i_{k}} & F_{A} \ar[l] \ar[d]^{i_{g}} \\
C_{k-1} & C_{k} \ar[l] & F_{C} \ar[l]
}
$$
A point in $F_{C}$ is defined by a pair $(g_{k+1\,;\,c}\,;\,g_{k+1\,;\,o})$ satisfying the relations of Definition \ref{Defifi}. Since $g_{k+1\,;\,c}$ is in the fiber over $i_{k-1}(g)$, the map sends all the faces of $\square(k+1\,;\,c)\times [0\,,\,1]$ on $\eta(\ast_{k+1\,;\,c})$ except the face $\square(k+1\,;\,c)\times \{1\}$. Furthermore they are no interaction between $g_{k+1\,;\,c}$ and $g_{k+1\,;\,o}$.\\
On the other hand the points in the image of $i_{g}$ coincide with the pair $(g_{k+1\,;\,c}\,;\,g_{k+1\,;\,o})$ such that:
$$
g_{k+1\,;\,c}:\square(k+1\,;\,c)\times [0\,,\,1]\rightarrow M(k+1\,;\,c)\,\,;\,\, (x\,;\,t)\rightarrow \eta(\ast_{k+1\,;\,c}).
$$ 
In order to prove that $i_{g}$ induces a deformation retract, we introduce the homotopy (also describe in \cite[Proposition 0.16]{MR1867354}) $H:\big( \square(k+1\,;\,c)\times [0\,,\,1]\big) \times [0\,,\,1]\rightarrow \square(k+1\,;\,c)\times [0\,,\,1]$ illustrated by the following picture: 
\begin{figure}[!h]
\begin{center}
\includegraphics[scale=0.4]{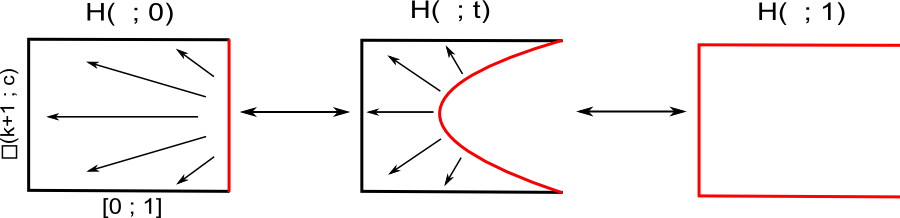}
\end{center}
\end{figure}\\
In other words, the points in the image of $i_{g}$ coincide with the pairs such that:
\begin{center}
$
g_{k+1\,;\,c}(x\,;\,t)=g_{k+1\,;\,c}\big( H\big(\,(x\,;\,t)\,;\,1\big)\,\big),
\,\,\,\,$ for $ x\in \square(k+1\,;\,c)$ and $t\in[0\,,\,1]$.
\end{center}
Finally the deformation retract is given by:
\begin{center}
$
\begin{array}{lcll}
H_{2}\,: & F_{C}\times [0\,,\,1] & \rightarrow & F_{C} \\ 
 & y=\big( (g_{k+1\,;\,c}\,;\,g_{k+1\,;\,o})\,;\,t_{1}\big) & \mapsto & 
 \left\{
 \begin{array}{lll}
 H_{2}(y)_{k+1\,;\,c}(x\,;\,t)  =  g_{k+1\,;\,c}\big( H\big((x\,;\,t)\,;\,t_{1}\big) \big) & $ for $ & x\in \square(k+1\,;\,c)$ and $ t\in [0\,,\,1], \\ 
 H_{2}(y)_{k+1\,;\,o}(x\,;\,1)  =  g_{k+1\,;\,o}(x\,;\,1) & $ for $ & x\in \square(k+1\,;\,o).
 \end{array} 
 \right.
\end{array} 
$
\end{center}
The space $\Omega \big( Bimod_{As_{>0}}(\square_{c} \,;\, M_{c})\,;\, Bimod_{\sc0}(\square \,;\, M)\big)$ is weakly equivalent to $Bimod_{\sc0}(\square\,;\, M^{\ast})$.
\end{proof}


\section{Double relative delooping: a particular case}
First of all we recall that for any pointed continuous map $f:A\rightarrow X$, the homotopy fiber $hofib(f)$ and the loop space $\Omega X$ based on $\ast$ are weakly equivalent to the pullback diagram $(I)$ and $(II)$:
\begin{center}
$
\xymatrix@R=0.2cm{
 & Top\big( [0\,,\,1]\,;\,X\big) \ar[dd]_{(ev_{0}\,;\,ev_{1})} & \\
 & & (I) \\
 \ast \times A \ar[r]_{id\times f} & X\times X & 
}\,\,\,\,\,\,\,\,\,\,\,\,\,\,\,\,\,\,\,\,\,\,\,\,\,\,\,\,\,\,
\xymatrix@R=0.2cm{
 & Top\big( [0\,,\,1]\,;\,X\big) \ar[dd]_{(ev_{0}\,;\,ev_{1})} & \\
 & & (II) \\
 \ast \times \ast \ar[r] & X\times X & 
}
$
\end{center}
By the double loop space $\Omega^{2}(X\,;\,A)$ we mean the loop space of the homotopy fiber $hofib(f)$. Since finite colimits commute, the double loop space can also be defined by the homotopy fiber of the continuous map $\Omega f$.

From now on, let $O$ be a multiplicative operad, that is, there exists an operad map $\alpha : As\rightarrow O$. Let $B$ be an $O$-bimodule equipped with an $O$-bimodule map $\beta : O\rightarrow B$. If we assume that $B(0)\simeq \ast$ we know from \cite[Theorem $6.2$]{Tourtchine:arXiv1012.5957} and the $As_{>0}$-bimodule map $\beta\,\circ\,\alpha :As\rightarrow B$ that $sTot(B)$ is weakly equivalent to the loop space $\Omega Bimod_{As_{>0}}^{h}(As_{>0}\,;\,B)$. Since $B$ is not an operad we can not expect that its semi-totalization has the homotopy type of a double loop space. However we will prove that $Bimod_{As_{>0}}^{h}(As_{>0}\,;\,B)$ has the homotopy type of a relative loop space by building an $\{o\,;\,c\}$-operad $X$ from the pair $(O\,;\,B)$:
\begin{equation}
X(n\,;\,c)=O(n),\,\,\,\,\, \text{for} \,\,n\geq 0\,\,\,; \,\,\,X(n\,;\,o)=B(n-1),\,\,\,\,\, \text{for}\,\, n>0, 
\end{equation} 
and the empty set otherwise. The operadic structure is defined by:
\begin{center}
$\begin{array}{llll}
\circ_{i}\,:\, X(n\,;\,c)\times X(m\,;\,c)\rightarrow X(n+m-1\,;\,c) & ; & (x\,;\,y)\mapsto x\,\circ_{i}\,y & $ using the operadic structure of  $O, \\ 
\circ_{i}\,:\, X(n\,;\,o)\times X(m\,;\,c)\rightarrow X(n+m-1\,;\,o) & ; & (x\,;\,y)\mapsto x\,\circ^{i}\,y & $ using the right $O$-bimodule structure of $ B, \\ 
\circ_{n}\,:\, X(n\,;\,o)\times X(m\,;\,o)\rightarrow X(n+m-1\,;\,o) & ; & (x\,;\,y)\mapsto \alpha(\ast_{2})(x\,;\,y) & $ using the left $O$-bimodule structure of $ B.
\end{array}$
\end{center}
The $\{o\,;\,c\}$-operad $X$ is endowed with a map of operads $\eta:\tc\rightarrow X\,\,;\,\,
\left\{
\begin{array}{l}
\eta(\ast_{i\,;\,c})=\alpha(\ast_{i}) \\ 
\eta(\ast_{i\,;\,o})=\beta\,\circ\,\alpha (\ast_{i-1})
\end{array} 
\right.
$.\\
The operadic axioms are satisfied except the unit axiom. This axiom holds under the assumption:
\begin{equation}
\alpha(\ast_{2})\big(\beta\circ\alpha(\ast_{0})\,;\,x\big)=\alpha(\ast_{2})\big(x\,;\,\beta\circ\alpha(\ast_{0})\big)=x \,\,\,\text{for}\,\,\, x\in X(m\,;\,o).
\end{equation}

\begin{thm}
Under Assumption $(13)$, the relative loop space $\Omega\big(  Operad^{h}(As_{>0}\,;\,O) \,;\, Operad_{\{o\,;\,c\}}^{h}(\sc0\,;\,X)\big)$ is weakly equivalent to $Bimod_{As_{>0}}^{h}(As_{>0}\,;\,B)$.    
\end{thm}

\begin{proof}
It is a consequence of Proposition \ref{LL1} and Proposition \ref{LL2}.
\end{proof}

\begin{defi} In order to describe the homotopy fiber the map $(9)$ of Definition \ref{Remarque} we need a cofibrant replacement of $\sc0$ as a coloured operad. 
Since $\sc0$ is cofibrant as an $\{o\,;\,c\}$-sequence, we know from \cite{2005math......2155B} that the Boardman-Vogt resolution of $\sc0$, denoted by $BV(\sc0)$ or just $\mathcal{WA}$ in our case, is the object we are looking for. We recall the construction:\\{}\\
$\bullet\,\,$ Let $tree_{n}^{o}$ be the subset of $\{o\,;\,c\}$-trees consisting of trees $(t\,,\,f)$ with $n$-leaves, $f$ is an $\{o\,;\,c\}$-labelling of $t$ and where the trunk is labelled by $o$, satisfying:
\begin{center}
$\forall v\in V(t)\,:\,
\left\{
\begin{array}{ccc}
  f(e_{0}(v))=c & \Rightarrow & f(e_{i}(v))=c\,\,\forall i\in\{1,\ldots , |v|\} \\ 
  f(e_{0}(v))=o & \Rightarrow & |v|>0\,,\, f(e_{|v|}(v))=o $ and $ f(e_{i}(v))=c \,\, \forall i\in \{1,\ldots,|v|-1\}
\end{array} 
\right.
$.
\end{center}
$\bullet\,\,$ The operad $\mathcal{WA}$ is the $\{o\,;\,c\}$-sequence given by: 

$$
\mathcal{WA}(n\,;\,c):=\,\,\,\, \left. \underset{t\,\in\, \{c\}-tree}{\coprod}\,\,\,\,\,\,\,\,\, \underset{v\,\in\, V(t)}{\prod}\, \sc0\big( f(e_{1}(v),\ldots,f(e_{|v|}(v))\,;\,f(e_{0}(v))\big)\,\times \, \underset{e\,\in\, E^{int}(t)}{\prod}\, [0\,,\,1]\right/ \sim
$$
$$
\mathcal{WA}(n\,;\,o):= \left. \,\,\,\,\,\, \underset{t\,\in\, tree_{n}^{o}}{\coprod}\,\,\,\,\,\,\,\,\,\,\,\,\, \underset{v\,\in\, V(t)}{\prod}\, \sc0\big( f(e_{1}(v),\ldots,f(e_{|v|}(v))\,;\,f(e_{0}(v))\big)\,\times \, \underset{e\,\in\, E^{int}(t)}{\prod}\, [0\,,\,1]\right/ \sim
$$
and the empty set otherwise. The equivalence relation $\sim$ is generated by contracting the inner edges indexed by $0$, using the operadic structure of $\sc0$, and the relation:
\begin{center}
\begin{figure}[!h]
\begin{center}
\includegraphics[scale=0.4]{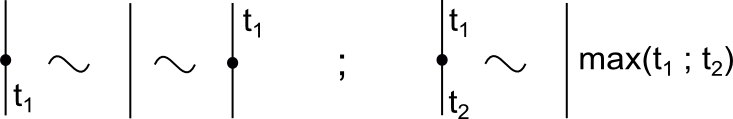}
\end{center}
\end{figure}
\end{center}
A point in $\mathcal{WA}(n\,;\,o)$ will be denoted by $[T\,;\,\{t_{e}\}]$ where $T$ is an element in $tree_{n}^{o}$ such that each vertex has at least  two inputs and $t_{e}\in [0\,,\,1]$ for each inner edge $e\,\in\, E^{int}(t)$. Similarly a point in $\mathcal{WA}(n\,;\,c)$ is denoted by $[T\,;\,\{t_{e}\}]$ with $T$ an element in $\{c\}$-tree. We will use the notation $v_{1}<v_{2}$ if $v_{1}\neq v_{2}$ are two connected vertices such that $d(v_{1}\,;\,r)<d(v_{2}\,;\,r)$.\\The operadic composition $\circ_{i}$ of two points $[T\,;\,\{t_{e}\}]$ and $[T'\,;\,\{t'_{e}\}]$ consists in grafting the tree $T'$ to the $i$-th leaf of $T$ and labelling the new inner edge by $1$.\\
$\bullet\,\,$ It is well known that the operad $\pent := \{\pent(n)=\mathcal{WA}(n\,;\,c)\}_{n>0}$ is a cofibrant replacement of $As_{>0}$ as an operad. It is usually called \textit{the Stasheff operad}.
\end{defi}

The operad $\mathcal{WA}$ has been introduced in \cite{2013arXiv1312.7155H} in order to recognize $A_{\infty}$-spaces and $A_{\infty}$-maps. The next definition is a description of the relative loop space  $\Omega\big( Operad^{h}(As_{>0}\,;\,X_{c})\,;\,Operad^{h}_{\{o\,;\,c\}}(\sc0\,;\,X)\big)$ using the cofibrant replacement $\mathcal{WA}$: 

\begin{defi}\label{defififi}
Define $\Omega \big( Operad(\pent\,;\,X_{c})\,;\,Operad_{\{o\,;\,c\}}(\mathcal{WA}\,;\,X)\big)$ to be the space of applications:
\begin{center}
$g_{n\,;\,c}:\mathcal{WA}(n\,;\,c)\times [0\,,\,1]\rightarrow X(n\,;\,c),\,\,\,\,\,$ for $n>0\,\,\,\,\,$ and  $\,\,\,\,\,g_{n\,;\,o}:\mathcal{WA}(n\,;\,o)\times \{1\}\rightarrow X(n\,;\,o),\,\,\,\,\,$ for $n>0$
\end{center}
satisfying the relations:\\{}\\
$
\begin{array}{lll}
\bullet & g_{n\,;\,c}(x\,\circ_{i}\,y\,;\,t)=g_{l+1\,;\,c}(x\,;\,t)\,\circ_{i}\,g_{n-l\,;\,c}(y\,;\,t) & \text{for}\,\,x\in \mathcal{WA}(l+1\,;\,c),\,\,y\in \mathcal{WA}(n-l\,;\,c)\,\, \text{and} i\in\{1,\ldots,l+1\},\\ 
\bullet & g_{n\,;\,o}(x\,\circ_{i}\,y\,;\,1)=g_{l+1\,;\,o}(x\,;\,1)\,\circ_{i}\,g_{n-l\,;\,c}(y\,;\,1) & \text{for} \,\, x\in \mathcal{WA}(l+1\,;\,o),\,\,y\in \mathcal{WA}(n-l\,;\,c) \,\,\text{and}\,\, i\in\{1,\ldots,l\},\\ 
\bullet & g_{n\,;\,o}(x\,\circ_{l+1}\,y\,;\,1)=g_{l+1\,;\,o}(x\,;\,1)\,\circ_{l+1}\,g_{n-l\,;\,o}(y\,;\,1) & \text{for} \,\,x\in \mathcal{WA}(l+1\,;\,o) \,\,\text{and}\,\,y\in \mathcal{WA}(n-l\,;\,o),
\end{array} 
$\\{}\\
and the boundary condition: $g_{n\,;\,c}(x\,;\,0)=\eta(\ast_{n\,;\,c})$ for $x\in \mathcal{WA}(n\,;\,c)$. 
\end{defi}

\begin{notat}
Let $\alpha : As_{>0}\rightarrow O$ be a map of operads, $\beta:O\rightarrow B$ be a map of $O$-bimodules and $\eta:\sc0\rightarrow X$ the corresponding map of $\{o\,;\,c\}$-operads. The $\{o\,;\,c\}$-sequence $X^{\ast}$  given by:
\begin{center}
 $X^{\ast}(n\,;\,c)=\eta(\ast_{n\,;\,c})\,\,\,$ for $n\geq 0\,$; $\,\,\,\,\,\,\,\,\,X^{\ast}(n\,;\,o)=B(n-1)\,\,\,$ for $n>0$
\end{center} 
and the empty set otherwise inherits from $X$ an $\{o\,;\,c\}$-operadic structure endowed with a map $\eta:\sc0\rightarrow X^{\ast}$.
\end{notat}

\begin{pro}\label{LL1}
Under Assumption $(13)$, the space $Bimod^{h}_{As_{>0}}(As_{>0}\,;\,B)$ is weakly equivalent to $Operad^{h}_{\{o\,;\,c\}}(\sc0\,;\, X^{\ast})$.
\end{pro}

\begin{proof} 
By assumption $B$ is an $As_{>0}$-bimodule. The first step of the proof consists in building a cofibrant replacement $\penta$ of $As_{>0}$ as an $As_{>0}$-bimodule such that there exists a map $\xi:Operad_{\{o\,;\,c\}}(\mathcal{WA}\,;\,X^{\ast})\rightarrow Bimod_{As_{>0}}(\penta\,;\,B)$.
Let us recall that a point $g\in Operad(\mathcal{WA}\,;\,X^{\ast})$ is described by:
\begin{center}
$
\left\{
\begin{array}{llll}
g_{n\,;\,c}: & \mathcal{WA}(n\,;\,c)\rightarrow X^{\ast}(n\,;\,c); & x\mapsto \eta(\ast_{n\,;\,c}), & $ for $ n>0, \\ 
g_{n\,;\,o}: & \mathcal{WA}(n\,;\,o)\rightarrow X^{\ast}(n\,;\,o), &  & $ for $ n>0.
\end{array} 
\right.
$
\end{center}
satisfying in particular for $x\in \mathcal{WA}(l+1\,;\,o)$, $y\in \mathcal{WA}(n-l\,;\,c)$ and $1\leq i \leq l$ the relation:
\begin{equation}
g_{n\,;\,o}(x\,\circ_{i}\,y)=g_{l+1\,;\,o}(x)\,\circ_{i}\,g_{n-l\,;\,c}(y)=g_{l+1\,;\,o}(x)\,\circ_{i}\,\eta(\ast_{n-l\,;\,c}).
\end{equation}
Define $\approx$ to be the equivalence relation on $\mathcal{WA}(n\,;\,o)$ generated by:
\begin{center}
$[T\,;\,\{t_{e}\}]\approx [T\,;\,\{l_{e}\}]\,\,\,\Leftrightarrow \,\,\, 
\left\{\begin{array}{l}
\bullet\,\, t_{e}=l_{e}\,\,\,\forall e\in E^{int}(T) $ with  $f(e)=o \\ 
$ and $ \\ 
\bullet\,\,  t_{e}=l_{e} \,\, $ if $\nexists\, e_{1}<e\,\, $ such that $ \,\, t_{e_{1}}=l_{e_{1}}=1$ and $ f(e_{1})=c.
\end{array}\right. 
$
\end{center}
We will denote by $\penta$ the sequence $\{\penta(n)=\mathcal{WA}(n+1\,;\,o)/\approx\}_{n>0}$. By convention $\penta(0)$ is the empty set.\\ 
Due to Relation $(14)$, the map $g$ induces a sequence map $\tilde{g}:=\{\tilde{g}_{n}:\penta (n)\rightarrow B(n)\}_{n>0}$.\\

Let us prove that $\penta$ is a cofibrant replacement of $As_{>0}$ as an $As_{>0}$-bimodule. The bimodule structure is given by:
\begin{itemize}
\item[$i)$] $-\,\circ^{i}\ast_{2}\,:\,\,\,\,\,\penta(n)\rightarrow \penta(n+1)\,\,\,\,\,\,\,\,\,\,\,\,\,\,\,\,\,\,\,\,\,\,;\,\,\,\, [T\,;\,\{t_{e}\}]\mapsto [T\,;\,\{t_{e}\}]\,\circ_{i}\,\delta_{2\,;\,c}\,\,\,\,\,\,\,$, for $1\leq i \leq n$,
\item[$ii)$] $\ast_{2}(-\,;\,-)\,:\,\penta(n)\times \penta(m)\rightarrow \penta(n+m) \,;\,\,\,\, \big([T_{1}\,;\,\{t_{e}\}]\,\,;\,\,[T_{2}\,;\,\{l_{e}\}]\big)\mapsto [T_{1}\,;\,\{t_{e}\}]\,\circ_{n+1}\,[T_{2}\,;\,\{l_{e}\}]$. 
\end{itemize}
where $\delta_{n\,;\,c}$ is the $n$-corolla in $\{c\}$-trees and $\delta_{n\,;\,o}$ is the $n$-corolla in $tree_{n}^{o}$. This structure satisfies the bimodule axioms over $As_{>0}$ and it makes $\tilde{f}$ into an $As_{>0}$-bimodule map. Furthermore $\penta$ is a cofibrant replacement:\\{}\\
\textbf{Cofibrant:} let $\penta_{n}$ be the $As_{>0}$-bimodule generated by $\{\penta(i)\}_{i=1}^{n}$ for $n>0$. By convention $\penta_{0}$ is the $As_{>0}$-bimodule $B_{As_{>0}}(\emptyset)$. Let us notice that the map $\mathcal{WA}(n+1\,;\,o)\rightarrow \penta(n)$ preserves the boundary and by definition a point in $\partial \mathcal{WA}(n+1\,;\,o)$ has one of the following form:
\begin{center}
$
\begin{array}{l}
\bullet\,\,[T\,;\,\{t_{e}\}] $ such that there exists $e_{1}\in E^{int}(T)$ with $t_{e_{1}}=1$ and $f(e_{1})=o \\ 
\bullet\,\,[T\,;\,\{t_{e}\}] $ such that there exists $e_{1}\in E^{int}(T)$ with $t_{e_{1}}=1$ and $f(e_{1})=c
\end{array} 
$
\end{center}
In the first case $[T\,;\,\{t_{e}\}]$ has a decomposition $[T_{1}\,;\,\{t^{1}_{e}\}]\,\circ_{|T_{1}|}\,[T_{2}\,;\,\{t^{2}_{e}\}]$. The image  lies in $\penta_{n-1}$ by the axiom $(ii)$. In the second case, $[T\,;\,\{t_{e}\}]$ has a similar decomposition $[T_{1}\,;\,\{t^{1}_{e}\}]\,\circ_{i}\,[T_{2}\,;\,\{t^{2}_{e}\}]$ with $i< |T_{1}|$, $|T_{2}|>1$ and we have the identification:
$$
[T\,;\,\{t_{e}\}]=[(T_{1}\,;\,\{t^{1}_{e}\})\,\circ_{i}\,(T_{2}\,;\,\{t^{2}_{e}\})]=[(T_{1}\,;\,\{t^{1}_{e}\})\,\circ_{i}\,\delta_{|T_{2}|\,;\,c}]=[(T_{1}\,;\,\{t^{1}_{e}\})\,\circ_{i}\,\delta_{|T_{2}|-1\,;\,c}]\,\circ^{i}\,\ast_{2},
$$
hence lies in $\penta_{n-1}$. Consequently $\penta_{n}$ is obtained from $\penta_{n-1}$ by the pushout diagram:
\begin{equation}
\xymatrix{
B_{As_{>0}}(\partial A) \ar[r] \ar[d]_{\tilde{q}} & B_{As_{>0}}(A) \ar[d] \\
\tilde{\pentagon}_{n-1} \ar[r] & \tilde{\pentagon}_{n}
}
\end{equation}
where $A$ is the sequence given by $A(n)=\mathcal{WA}(n+1\,;\,o)$ and the empty set otherwise. The attaching map is the restriction of the quotient map $q:\mathcal{WA}(n+1\,;\,o)\rightarrow \penta(n)$ to the boundary. Furthermore if $i\geq n$ then $\penta_{i}(n)=\penta(n)$ and the map $\partial A \rightarrow A$ is a cofibration. So $lim_{i}\,\penta_{i}=\penta$ and $\penta$ is cofibrant. Lie in the proof of Proposition \ref{Lmm1TF}, these sequences of pushout diagram imply that the spaces $\penta(n)$ are $CW$-complex for each $n$.\\{}\\
\textbf{Contractible:} The map $q:\mathcal{WA}(n+1\,;\,o)\rightarrow \penta(n)$ is a continuous map between compact $CW$-complexes. Since the fiber of $q$ over a point is homeomorphic to a product of polytopes which is contractible, the map $q$ is a weak equivalence  \cite[Main Theorem]{Smale}. Hence $\penta(n)$ is contractible for $n>0$.\\{}\\

Since $\penta$ is a cofibrant replacement of $As_{>0}$ as a bimodule over itself, the space $Bimod^{h}_{As_{>0}}(As_{>0}\,;\,B)$ is weakly equivalent to $Bimod_{As_{>0}}(\penta\,;\,B)$ and the assignment $\xi(g)=\tilde{g}$ defines a map:
\begin{center}
$
\xi:Operad_{\{o\,;\,c\}}(\mathcal{WA}\,;\,X^{\ast})\rightarrow Bimod_{As_{>0}}(\penta\,;\, B)\,\,;\,\,g\mapsto \tilde{g}.
$
\end{center}
In order to prove that $\xi$ is a weak equivalence, we introduce two towers of fibrations.
Define $A'_{k}$ and $B'_{k}$ to be the subspaces:
$$A'_{k}\subset \underset{\text{reduced to a point}}{\underbrace{\underset{i=1}{\overset{k+1}{\prod}}\, Top\big( \mathcal{WA}(i\,;\,c)\,;\, X^{\ast}(i\,;\,c)\big)}}\times \underset{i=1}{\overset{k+1}{\prod}}\, Top\big( \mathcal{WA}(i\,;\,o)\,;\,X^{\ast}(i\,;\,o)\big)\,\,\,\,\,\,\text{and}\,\,\,\,\,B'_{k}\subset \underset{i=1}{\overset{k}{\prod}}\, Top\big( \tilde{\pentagon}(i)\,;\, B(i)\big)$$
with $A'_{k}$ satisfying the operadic relations and $B'_{k}$ the $As_{>0}$-bimodule relations for $k>0$. In other words $A'_{k}$ and $B'_{k}$ are respectively the space $Operad(\mathcal{WA}_{k+1}\,;\,X^{\ast})$ and $Bimod_{As_{>0}}(\penta_{k}\,;\,B)$ where $\mathcal{WA}_{k+1}$ is the sub-operad of $\mathcal{WA}$ generated by $\{\mathcal{WA}(i\,;\,c)\}_{i=1}^{k+1}$ and $\{\mathcal{WA}(i\,;\,o)\}_{i=1}^{k+1}$. Since $\mathcal{WA}(1\,;\,c)$ and $\mathcal{WA}(1\,;\,o)$ are reduced to the unit, the factors $Top\big( \mathcal{WA}(1\,;\,c)\,;\,X^{\ast}(1\,;\,c)\big)$ and $Top\big( \mathcal{WA}(1\,;\,o)\,;\,X^{\ast}(1\,;\,o)\big)$ are one point spaces and can be ignored.
So we consider the two towers:
\begin{center}
$
\xymatrix@R=0.1cm{
A'_{1} & A'_{2} \ar[l] & \cdots \ar[l] & A'_{k} \ar[l] & A'_{k+1} \ar[l] & \cdots \ar[l] \\
B'_{1} & B'_{2} \ar[l] & \cdots \ar[l] & B'_{k} \ar[l] & B'_{k+1} \ar[l] & \cdots \ar[l]
}
$
\end{center}
so that:\\{}\\
$
\begin{array}{lllllll}
A'_{\infty} & = & lim_{k}\, A'_{k} & \simeq & holim_{k}\, A'_{k} & \simeq & Operad_{\{o\,;\,c\}}(\mathcal{WA}\,;\,X^{\ast}), \\ 
B'_{\infty} & = & lim_{k}\, B'_{k} & \simeq & holim_{k}\, B'_{k} & \simeq & Bimod_{As_{>0}}(\penta\,;\,B).
\end{array} 
$\\{}\\
By restriction, the map $\xi$ induces an application between the two towers:
$$
\xymatrix{
A'_{1} \ar[d]^{\xi_{1}} & A'_{2} \ar[l] \ar[d]^{\xi_{2}} & \cdots \ar[l] & A'_{k} \ar[l] \ar[d]^{\xi_{k}} & A'_{k+1} \ar[l] \ar[d]^{\xi_{k+1}}& \cdots \ar[l] \\
B'_{1} & B'_{2} \ar[l] & \cdots \ar[l] & B'_{k} \ar[l] & B'_{k+1} \ar[l] & \cdots \ar[l]
}
$$
with $\xi\,=\, lim_{k}\,\xi_{k}\, = \, holim_{k}\, \xi_{k}$. Consequently, $\xi$ is a weak equivalence if each $\xi_{k}$ is a weak equivalence. We will prove this result by induction on $k$:\\
 $\bullet\,\,\,\xi_{1}$ coincides with the identity. It is a weak equivalence.\\
$\bullet\,\,$ Assume that $\xi_{k-1}$ is a weak equivalence. We consider the following diagram where $g$ is a point in $A'_{k-1}$, $F_{A'}$ is the fiber over $g$ and $F_{B'}$ the fiber over $\xi_{k-1}(g)$. Since the two left horizontal arrows are fibrations, the map $\xi_{k}$ is a weak equivalence if the induced map $\xi_{g}$ is a weak equivalence.
$$
\xymatrix{
A'_{k-1} \ar[d]^{\xi_{k-1}} & A'_{k} \ar[l] \ar[d]^{\xi_{k}} & F_{A'} \ar[l] \ar[d]^{\xi_{g}} \\
B'_{k-1} & B'_{k} \ar[l] & F_{B'} \ar[l]
}
$$
From Lemma \ref{lemme2} the fiber $F_{A'}$ is homeomorphic to the space   $Top^{g_{k+1\,;\,o}}\big(\,\big( \mathcal{WA}(k+1\,;\,o)\,;\,\partial \mathcal{WA}(k+1\,;\,o)\big)\,;\, B(k)\big)$. Similarly $\penta_{k}$ is obtained from $\penta_{k-1}$ by the pushout diagram $(15)$. So the fiber $F_{B'}$ is homeomorphic to the space $Top^{\xi_{k-1}(g)_{k}\circ q}\big(\,\big( \mathcal{WA}(k+1\,;\,o)\,;\,\partial \mathcal{WA}(k+1\,;\,o)\big)\,;\, B(k)\big)$ and we have the commutative square:
$$
\xymatrix{
F_{A'} \ar[d]^{\xi_{g}} \ar[r] & Top^{g_{k+1\,;\,o}}\big(\,\big( \mathcal{WA}(k+1\,;\,o)\,;\,\partial \mathcal{WA}(k+1\,;\,o)\big)\,;\, B(k)\big) \ar@{=}[d]^{id}\\
F_{B'} \ar[r] & Top^{\xi_{k-1}(g)_{k}\circ q}\big(\,\big( \mathcal{WA}(k+1\,;\,o)\,;\,\partial \mathcal{WA}(k+1\,;\,o)\big)\,;\, B(k)\big)
}
$$
consequently $\xi_{k}$ is a weak equivalence.
\end{proof}

\begin{pro}\label{LL2}
Under Assumption $(13)$, the relative loop space $\Omega\big( Operad^{h}(As_{>0}\,;\,X_{c})\,;\,Operad^{h}_{\{o\,;\,c\}}(\sc0\,;\,X)\big)$ is weakly equivalent to the space $Operad^{h}_{\{o\,;\,c\}}(\sc0\,;\,X^{\ast})$. 
\end{pro}

\begin{proof}
We can consider $Operad_{\{o\,;\,c\}}(\mathcal{WA}\,;\,X^{\ast})$ as a subspace of $\Omega \big( Operad(\pent \,;\, X_{c})\,;\, Operad_{\{o\,;\,c\}}(\mathcal{WA}\,;\,X)\big)$ using the inclusion:
$$
\begin{array}{cccc}
i\,: & Operad_{\{o\,;\,c\}}(\mathcal{WA}\,;\,X^{\ast}) & \rightarrow & \Omega \big( Operad(\pent \,;\, X_{c})\,;\, Operad_{\{o\,;\,c\}}(\mathcal{WA}\,;\,X)\big) \\ 
 & g & \mapsto & 
 \left\{\begin{array}{cccc}
 \tilde{g}_{n\,;\,c}\,: & \mathcal{WA}(n\,;\,c)\times [0\,,\,1] \rightarrow X(n\,;\,c) & ; & (x\,;\,t)\mapsto \eta(\ast_{n\,;\,c}) \\ 
 \tilde{g}_{n\,;\,o}\,: & \mathcal{WA}(n\,;\,o)\times \{1\} \rightarrow X(n\,;\,o)  & ; &  (x\,;\,1)\mapsto g_{n\,;\,o}(x)
 \end{array}\right. 
\end{array} 
$$
In order to show that $i$ is a weak equivalence, we introduce two towers of fibrations. One of them is the tower $A'_{k}$ of Proposition \ref{LL1}. The second one is defined by:
$$
C'_{k}\subset \underset{i=1}{\overset{k+1}{\prod}}\, Top\big( \mathcal{WA}(i\,;\,c)\times [0\,,\,1]\,;\, X(i\,;\,c)\big)\times \underset{i=1}{\overset{k+1}{\prod}}\, Top\big( \mathcal{WA}(i\,;\,o)\,;\,X(i\,;\,o)\big)
$$
satisfying the relations of Definition \ref{defififi}. Since $\mathcal{WA}(1\,;\,c)$ and $\mathcal{WA}(1\,;\,o)$ are reduced to the unit, the factors $Top\big(\mathcal{WA}(1\,;\,c)\times [0\,,\,1]\,;\,X(1\,;\,c)\big)$ and $Top\big(\mathcal{WA}(1\,;\,o)\,;\,X(1\,;\,o)\big)$ are the one point space and can be ignored. The restriction to the space $A'_{k}$ of the inclusion $i$ induces a map between the two towers:
$$
\xymatrix{
A'_{1} \ar[d]^{i_{1}} & A'_{2} \ar[l] \ar[d]^{i_{2}} & \cdots \ar[l] & A'_{k} \ar[l] \ar[d]^{i_{k}} & A'_{k+1} \ar[l] \ar[d]^{i_{k+1}}& \cdots \ar[l] \\
C'_{1} & C'_{2} \ar[l] & \cdots \ar[l] & C'_{k} \ar[l] & C'_{k+1} \ar[l] & \cdots \ar[l]
}
$$
Since the space $\Omega \big( Operad(\pent \,;\, X_{c})\,;\, Operad_{\{o\,;\,c\}}(\mathcal{WA}\,;\,X)\big)$ is weakly equivalent to the limit of $C'_{k}$, the map $i$ is a weak equivalence if each $i_{k}$ is a weak equivalence. We will prove this result by induction on $k$:\\{}\\
$\bullet\,\,$ If $k=1$, a point in $C'_{1}$ is a pair $(g_{2\,;\,c}\,;\,g_{2\,;\,o})$ whereas the points in the image of $i_{1}$ coincide with the pairs satisfying:
$$
g_{2\,;\,c}:\mathcal{WA}(2\,;\,c)\times [0\,,\,1]\rightarrow X(2\,;\,c)\,\,;\,\,(x\,;\,t)\mapsto \eta(\ast_{2\,;\,c}).
$$
Since $g_{2\,;\,c}(x\,;\,0)=\eta(\ast_{1\,;\,c})$ for any pair in $C'_{1}$, the inclusion $i_{1}$ induces the following deformation retract:
\begin{center}
$
\begin{array}{cccl}
H\,: & C'_{1}\times[0\,,\,1] & \rightarrow & C'_{1} \\ 
 & y=\big( (g_{2\,;\,c}\,;\,g_{2\,;\,o})\,;\,t_{1}\big) & \mapsto & 
 \left\{
 \begin{array}{l}
 H(y)_{2\,;\,c}(\delta_{2\,;\,c}\,;\,t)=g_{2\,;\,c}\big(\delta_{2\,;\,c}\,;\,t(1-t_{1})\big) \\ 
 H(y)_{2\,;\,o}(\delta_{2\,;\,o}\,;\,1)=g_{2\,;\,o}(\delta_{2\,;\,o}\,;\,1)
 \end{array} 
 \right.
\end{array} 
$
\end{center}
$\bullet\,\,$ From now on we assume that $i_{k-1}$ is a weak equivalence for $k\geq 2$. We consider the following diagram where $g$ is a point in $A'_{k-1}$, $F_{A'}$ is the fiber over $g$ and $F_{C'}$ the fiber over $i_{k-1}(g)$. Since the two left horizontal arrows are fibrations, the map $i_{k}$ is a weak equivalence if the induced map $i_{g}$ is a weak equivalence. 
$$
\xymatrix{
A'_{k-1} \ar[d]^{i_{k-1}} & A'_{k} \ar[l] \ar[d]^{i_{k}} & F_{A'} \ar[l] \ar[d]^{i_{g}} \\
C'_{k-1} & C'_{k} \ar[l] & F_{C'} \ar[l]
}
$$
A point in the fiber $F_{C'}$ is defined by a pair $(g_{k+1\,;\,c}\,;\,g_{k+1\,;\,o})$ satisfying the relations of Definition \ref{defififi}. Since the pair is in the fiber over $i_{k-1}(g)$, the map $g_{k+1\,;\,c}$ sends all the faces of $\mathcal{WA}(k+1\,;\,c)\times [0\,,\,1]$ on $\eta(\ast_{k+1\,;\,c})$ except for the face $\mathcal{WA}(k+1\,;\,c)\times \{1\}$.\\
On the other hand the points in the image of $i_{g}$ coincide with the pairs $(g_{k+1\,;\,c}\,;\,g_{k+1\,;\,o})$ such that:
$$
g_{k+1\,;\,c}:\mathcal{WA}(k+1\,;\,c)\times [0\,,\,1]\rightarrow X(k+1\,;\,c)\,\,;\,\, (x\,;\,t)\mapsto \eta(\ast_{k+1\,;\,c}).
$$
Let us denote by $Face$ the interior of the space $\mathcal{WA}(k+1\,;\,c)\times \{1\}$ such that the inclusion:
$$
\partial\big( \mathcal{WA}(k+1\,;\,c)\times [0\,,\,1]\big)\setminus Face\rightarrow \mathcal{WA}(k+1\,;\,c)\times [0\,,\,1]
$$
is an acyclic cofibration. In order to prove that $i_{g}$ induces a deformation retract, we consider a lift $H$ in the following diagram:
$$
\xymatrix{
\big( \partial \big( \mathcal{WA}(k+1\,;\,c)\times [0\,\,1]\big)\setminus Face\big) \times [0\,,\,1] \,\,\bigsqcup \,\, \big( \mathcal{WA}(k+1\,;\,c)\times [0\,,\,1]\big) \times \{0\} \ar[r] \ar[d]_{\simeq} & \mathcal{WA}(k+1\,;\,c)\times [0\,,\,1] \\
\big( \mathcal{WA}(k+1\,;\,c)\times [0\,,\,1]\big)\times [0\,,\,1] \ar@{-->}[ru]_{H} & 
}
$$
where the horizontal arrow is the inclusion on the factor $\big( \mathcal{WA}(k+1\,;\,c)\times [0\,,\,1]\big) \times \{0\}$ and sends a point $\big( (x\,;\,t_{1})\,;\,t_{2}\big)\in \big( \partial \big( \mathcal{WA}(k+1\,;\,c)\times [0\,\,1]\big)\setminus Face\big) \times [0\,,\,1]$ to $(x\,;\,(1-t_{2})t_{1})$. The homotopy $H$ is illustrated by the following picture:
\begin{center}
\begin{figure}[!h]
\begin{center}
\includegraphics[scale=0.4]{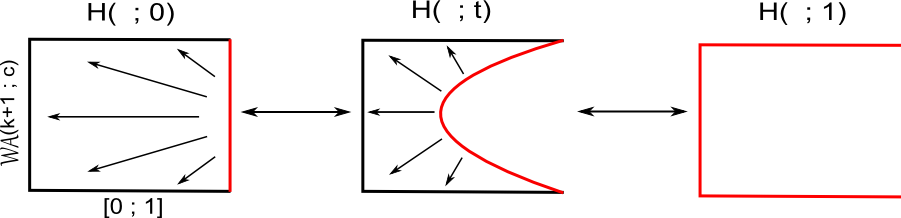}
\end{center}
\end{figure}
\end{center}
In other words, the points in the image of $i_{g}$ coincide with the pairs such that:
\begin{center}
$
g_{k+1\,;\,c}(x\,;\,t)=g_{k+1\,;\,c}\big( H\big(\,(x\,;\,t)\,;\,1\big)\,\big)=\eta(\ast_{k+1\,;\,c}),
\,\,\,\,$ for $ x\in \mathcal{WA}(k+1\,;\,c)$ and $t\in[0\,,\,1]$.
\end{center}
Finally the deformation retract is given by:
\begin{center}
$
\begin{array}{lcll}
H_{2}\,: & F_{C'}\times [0\,,\,1] & \rightarrow & F_{C'} \\ 
 & y=\big( (g_{k+1\,;\,c}\,;\,g_{k+1\,;\,o})\,;\,t_{1}\big) & \mapsto & 
 \left\{
 \begin{array}{ccc}
 H_{2}(y)_{k+1\,;\,c}(x\,;\,t)  =  g_{k+1\,;\,c}\big( H\big((x\,;\,t)\,;\,t_{1}\big) \big) & $ for $ & x\in \mathcal{WA}(k\,;\,c)$ and $ t\in [0\,,\,1] \\ 
 H_{2}(y)_{k+1\,;\,o}(x\,;\,1)  =  g_{k+1\,;\,o}(x\,;\,1) & $ for $ & x\in \mathcal{WA}(k+1\,;\,o)
 \end{array} 
 \right.
\end{array} 
$
\end{center}
Consequently $\Omega\big( Operad(\pent\,;\,X_{c})\,;\,Operad_{\{o\,;\,c\}}(\mathcal{WA}\,;\,X)\big)$ is weakly equivalent to $Operad_{\{o\,;\,c\}}(\mathcal{WA}\,;\,X^{\ast})$.
\end{proof}


\begin{cor}
Let $\alpha : As\rightarrow O$ be a map of operads and $\beta : O\rightarrow B$ be a map of $O$-bimodules. Under Assumption $(13)$, if $B(0)\simeq \ast$ and $O(0)\simeq O(1) \simeq \ast$ then the pair $(sTot(O)\,;\,sTot(B))$ is weakly equivalent to the $\mathcal{SC}_{2}$-space: 
\begin{center}
$\big( \Omega^{2}Operad^{h}(As_{>0}\,;\,O)\,;\,\Omega^{2}\big( Operad^{h}(As_{>0}\,;\,X_{c})\,;\,Operad^{h}_{\{o\,;\,c\}}(\sc0\,;\,X)\big)\,\big)$,
\end{center}
where $X$ is the operad given by Relations $(12)$.
\end{cor}

\section{Double relative delooping: general case}
In this section $O$ is an $\{o\,;\,c\}$-operad endowed with a map of operads $\eta:\tc\rightarrow O$ which makes $O$ into an $\sc0$-bimodule under $\tc$. We denote by $(O_{c}\,;\,O_{o})$ the pair of semi-cosimplicial spaces associated to $O$ (see Proposition \ref{propSC}). In  section $4$ we proved that the pair $(sTot(O_{c})\,;\,sTot(O_{o}))$ is weakly equivalent to:
\begin{center}
$\big( \Omega Bimod^{h}_{As_{>0}}(As_{>0}\,;\,O_{c})\,;\,\Omega \big( Bimod^{h}_{As_{>0}}(As_{>0} \,;\, O_{c})\,;\, Bimod^{h}_{\sc0}(\sc0 \,;\, O)\big)\,\big)$
\end{center}
under the assumption $O(0\,;\,c)\simeq \ast$.

If we assume that $O(1\,;\,c)\simeq \ast$, then $Bimod^{h}_{As_{>0}}(As_{>0}\,;\,O_{c})\simeq \Omega Operad^{h}(As_{>0}\,;\,O_{c})$. Similarly Marcy D. Robertson shows in \cite{2011arXiv1111.3904R} that the derived space of bimodule maps is weakly equivalent to the loop space of the derived space of  operadic maps. More precisely, in our case we have:

\begin{pro}\label{terminal}
Let $\eta:\tc\rightarrow O$ be a map of $\{o\,;\,c\}$-operads with $O(1\,;\,c)\simeq O(1\,;\,o) \simeq \ast$. The space $Bimod^{h}_{\sc0}(\sc0\,;\,O)$ is weakly equivalent to $\Omega Operad^{h}_{\{o\,;\,c\}}(\sc0\,;\,O)$. 
\end{pro}  

\begin{sproof}
The proof is the same as in \cite[Theorem $7.2$]{Tourtchine:arXiv1012.5957}. According to the notation of Turchin, $\mathcal{B}\pent$ is a cofibrant replacement of $As_{>0}$ in the model category $Bimod_{As_{>0}}$ so that there exists a map:
\begin{center}
$\xi_{c}:  \Omega Operad(\pent\,;\,O_{c}) \rightarrow Bimod_{As_{>0}}(\mathcal{B}\pent\,;\,O_{c})\,\,;\,\, f\mapsto \xi^{f}_{c}\,\,$.
\end{center}
Using towers of fibrations as in section $4$ and $5$, Turchin proves that $\xi_{c}$ is a weak equivalence. The construction of the map $\xi_{c}$ is obtained from a polytope subdivision $\mathcal{B}\pent(n)=\{\mathcal{B}\pent(T)\}_{T}$ indexed by $\{c\}$-trees with $n$ leaves. More precisely, for any $\{c\}$-tree $T$ with $n$ leaves the space $\mathcal{B}\pent(T)$ is the product of the two spaces:\\{}\\
$
\begin{array}{l}
\bullet\,\, \lambda_{\pentagon}(T)=\underset{v\in V(T)}{\prod}\,\pent(|v|),  \\ 
\bullet\,\, \chi_{\blacktriangle}(T)=\{\,\{t_{v}\}_{v\in V(T)}\,|\, t_{v}\in [0\,,\,1] $ and $ t_{v_{1}}\leq t_{v_{2}} $ if $ v_{1}<v_{2}\}\subset [0\,,\,1]^{|V(T)|}.
\end{array} 
$\\{}\\
A point in $\mathcal{B}\pent(T)$ is denoted by $\{x_{v}\,;\,t_{v}\}$ with $\{x_{v}\}\in \lambda_{\pentagon}(T)$ and $\{t_{v}\}\in \chi_{\blacktriangle}(T)$. For any $f\in \Omega Operad(\pent\,;\,O_{c})$, the map $\xi_{c}^{f}$ is defined on each polytope $\mathcal{B}\pent(T)$ by induction on the number of vertices of $T$ using the operadic structure of $O_{c}$:
\begin{center}
$
\begin{array}{cccl}
 \xi^{f}_{T\,;\,c}\, : & \mathcal{B}\pent(T) & \rightarrow & O(n\,;\,c) \\ 
 & \{x_{v}\,;\,t_{v}\} & \mapsto & f_{|r|}(x_{r}\,;\,t_{r})\,\big( \xi_{T_{1}\,;\,c}^{f}(\{x^{1}\,;\,t^{1}\}),\ldots, \xi_{T_{|r|}\,;\,c}^{f}(\{x^{|r|}\,;\,t^{|r|}\})\big)
\end{array} 
$
\end{center}
with $T_{i}$ the subtree of $T$ whose trunk coincides with the $i$-th input edge of the root of $T$.

In our case, a cofibrant replacement of $\sc0$ in the model category $Bimod_{\sc0}$ is the $\{o\,;\,c\}$-sequence:
\begin{center}
$\mathcal{B}\pent_{o\,;\,c}(n\,;\,c)=\mathcal{B}\pent_{o\,;\,c}(n\,;\,o)=\mathcal{B}\pent(n),\,\,\,\,\,$ for $n>0$ 
\end{center}
and the empty set otherwise with the obvious $\sc0$-bimodule structure. The space $\mathcal{B}\pent_{o\,;\,c}(n\,;\,o)$ has a polytope subdivision $\{\mathcal{B}\pent_{o\,;\,c}(T)\}_{T}$ indexed by $tree_{n}^{o}$. The space $\mathcal{B}\pent_{o\,;\,c}(T)$ is the product of the two spaces:\\{}\\
$
\begin{array}{l}
\bullet\,\, \lambda_{\pentagon}(T)=\underset{v\in V(T)}{\prod}\,\mathcal{WA}\big(n\,;\,f(e_{0}(v))\big),  \\ 
\bullet\,\, \chi_{\blacktriangle}(T)=\{\,\{t_{v}\}_{v\in V(T)}\,|\, t_{v}\in [0\,,\,1] $ and $ t_{v_{1}}<t_{v_{2}} $ if $ v_{1}<v_{2}\}\subset [0\,,\,1]^{|V(T)|}.
\end{array} 
$\\{}\\
A point in $\mathcal{B}\pent_{o\,;\,c}(T)$ is denoted by $\{x_{v}\,;\,t_{v}\}$. For any $f\in \Omega Operad(\mathcal{WA}\,;\,O)$, the map $\xi^{f}$ is defined by $\xi^{f}_{n\,;\,c}$ as previously and $\xi^{f}_{n\,;\,o}$ by induction on the number of vertices of $T$ using the operadic structure of $O$:
\begin{center}
$
\begin{array}{cccl}
 \xi^{f}_{T\,;\,o}\, : & \mathcal{B}\pent_{\{o\,;\,c\}}(T) & \rightarrow & O(n\,;\,o) \\ 
 & \{x_{v}\,;\,t_{v}\} & \mapsto & f_{|r|\,;\,o}(x_{r}\,;\,t_{r})\,\big( \xi_{T_{1}\,;\,c}^{f}(\{x^{1}\,;\,t^{1}\}),\ldots, \xi_{T_{|r|}\,;\,o}^{f}(\{x^{|r|}\,;\,t^{|r|}\})\big)
\end{array} 
$
\end{center}
with $T_{i}$ the subtree of $T$ whose trunk coincides with the $i$-th input edge of the root of $T$. It defines a map from  $\Omega Operad_{\{o\,;\,c\}}(\mathcal{WA}\,;\,O)$ to $Bimod_{\sc0}(\mathcal{B}\pent_{o\,;\,c}\,;\,O)$ which is a weak equivalence using the same arguments as Turchin in \cite{Tourtchine:arXiv1012.5957}. 
\end{sproof}

\begin{thm}
Assume $O$ is an $\{o\,;\,c\}$-operad such that $O(0\,;\,c)\simeq O(1\,;\,c)\simeq \ast$ and $O(1\,;\,o)\simeq \ast$. Let $\eta:\tc\rightarrow O$ be a map of $\{o\,;\,c\}$-operads. The pair $(sTot(O_{c})\,;\,sTot(O_{o}))$ is weakly equivalent to the $\mathcal{SC}_{2}$-space:
\begin{center}
$\big( \Omega^{2}Operad^{h}(As_{>0}\,;\,O_{c})\,;\,\Omega^{2}\big( Operad^{h}(As_{>0}\,;\,O_{c})\,;\,Operad^{h}_{\{o\,;\,c\}}(\sc0\,;\,O)\big)\,\big)$.
\end{center}
\end{thm}

\begin{proof}
By \cite[Theorem $7.2$]{Tourtchine:arXiv1012.5957} the space $sTot(O_{c})$ is weakly equivalent to $\Omega^{2}Operad(\pent\,;\,O_{c})$.
Proposition \ref{terminal} implies that the projection of $\xi$ onto the closed part gives rise to the commutative diagram:
\begin{center}
$
\xymatrix{
\Omega Operad_{\{o\,;\,c\}}(\mathcal{WA}\,;\,O) \ar[r]^{\xi} \ar[d]_{\Omega (p_{2})} & Bimod_{\sc0}(\square\,;\,O) \ar[d]^{p_{1}} \\
\Omega Operad(\pentagon\,;\,O_{c}) \ar[r]^{\xi_{c}} & Bimod_{As_{>0}}(\square_{c}\,;\,O_{c})
}
$
\end{center}
where $p_{1}$ and $p_{2}$ are respectively the maps $(8)$ and $(9)$.\\
Since the homotopy fibers commute with the homotopy limits, we have the following:
\begin{center}
$
\begin{array}{lll}
\Omega\big( Bimod_{As_{>0}}(\square_{c}\,;\,O_{c})\,;\,Bimod_{\sc0}(\square\,;\,O)\big) & \simeq & hofib\big( Bimod_{\sc0}(\square\,;\,O) \overset{p_{1}}{\longrightarrow} Bimod_{As_{>0}}(\square_{c}\,;\,O_{c})\big) \\ 
 & \simeq & hofib\big( \Omega Operad_{\{o\,;\,c\}}(\mathcal{WA}\,;\,O) \overset{\Omega(p_{2})}{\longrightarrow} \Omega Operad(\pent\,;\,O_{c})\big) \\ 
 & \simeq &  \Omega\,hofib\big( Operad_{\{o\,;\,c\}}(\mathcal{WA}\,;\,O) \overset{p_{2}}{\longrightarrow}  Operad(\pent\,;\,O_{c})\big)\\ 
 & \simeq & \Omega^{2} \big( Operad(\pent\,;\,O_{c})\,;\,Operad_{\{o\,;\,c\}}(\mathcal{WA}\,;\,O)\big)\,\big)
\end{array} 
$
\end{center}
\end{proof}

\bibliographystyle{amsplain}
\bibliography{bibliography}

\end{document}